\documentclass[10pt]{article}

\usepackage{url}
\usepackage{mathtools}
\usepackage{amssymb}
\usepackage{amsthm}
\usepackage{empheq}
\usepackage{latexsym}
\usepackage{enumitem}
\usepackage{eurosym}
\usepackage{dsfont}
\usepackage{appendix}
\usepackage{color} 
\usepackage[unicode]{hyperref}
\usepackage{frcursive}
\usepackage[utf8]{inputenc}
\usepackage[T1]{fontenc}
\usepackage{geometry}
\usepackage{multirow}
\usepackage[colorinlistoftodos]{todonotes}
\usepackage{lmodern}
\usepackage{anyfontsize}
\usepackage{stmaryrd}
\usepackage{natbib}
\usepackage{cleveref}

\usepackage{amsbsy}

\usepackage{comment}

\setcitestyle{numbers,open={[},close={]}}

\definecolor{red}{rgb}{0.7,0.15,0.15}
\definecolor{green}{rgb}{0,0.5,0}
\definecolor{blue}{rgb}{0,0,0.7}
\hypersetup{colorlinks, linkcolor={red},citecolor={green}, urlcolor={blue}}
			
\makeatletter \@addtoreset{equation}{section}

\newtheorem{theorem}{Theorem}[section]

\newtheorem{corollary}[theorem]{Corollary}

\newtheorem{proposition}[theorem]{Proposition}

\newtheorem{definition}[theorem]{Definition}
\newtheorem{remark}[theorem]{Remark}

\def\no{\noindent}
\def\beq{\begin{eqnarray}}
\def\eeq{\end{eqnarray}}
\def\be*{\begin{eqnarray*}}
\def\ee*{\end{eqnarray*}}

\def \E{\mathbb{E}}
\def \F{\mathbb{F}}

\def \H{\mathbb{H}}

\def \N{\mathbb{N}}

\def \P{\mathbb{P}}

\def \R{\mathbb{R}}

\def\Ac{{\cal A}}

\def\Cc{{\cal C}}

\def\Fc{{\cal F}}
\def\Gc{{\cal G}}
\def\Hc{{\cal H}}

\def\Lc{{\cal L}}

\def\Oc{{\cal O}}
\def\Pc{{\cal P}}

\def\Rc{{\cal R}}
\def\Sc{{\cal S}}

\def\Wc{{\cal W}}

\def\x{\times}

\def\Om{\Omega}

\def\om{\omega}

\def\0{\mathbf{0}}

\def \mub{\overline{\mu}}

\def \muh{\widehat{\mu}}

\def\normeL2#1{\left\|{#1}\right\|_{L^2}}

\def\Xh{\widehat X}

\def\esup{{\rm ess \, sup}}

\def \Ubb {\boldsymbol{U}}

\def \Xbb{\mathbf{X}}

\def \Qr{\mathrm{Q}}

\def \1{\mathds{1}}

\def \xbb {\boldsymbol{x}}

\def \Xbb{\mathbf{X}}

\def \Qr{\mathrm{Q}}

\DeclareUnicodeCharacter{014D}{\=o}
 \title{A notion of BSDE on the Wasserstein space and its applications to control problems and PDEs\footnote{The author thanks Dylan {\sc Possama\"{i}}  for his helpful suggestions and interesting discussions.}
 
 }

\author{
 Mao Fabrice Djete\footnote{\'Ecole Polytechnique Paris, Centre de Math\'ematiques Appliqu\'ees, mao-fabrice.djete@polytechnique.edu. This work benefits from the financial support of the Chairs {\it Financial Risk} and {\it Finance and Sustainable Development}} 
    }
             \date{\today}

\begin{document}

\maketitle
 
\begin{abstract}
    We introduce a class of backward stochastic differential equations (BSDEs) on the Wasserstein space of probability measures. This formulation extends the classical correspondence between BSDEs, stochastic control, and partial differential equations (PDEs) to the mean--field (McKean–Vlasov) setting, where the dynamics depend on the law of the state process. The standard BSDE framework becomes inadequate in this context, motivating a new definition in terms of measure--dependent solutions.

\medskip
Under suitable assumptions, we demonstrate that this formulation is in correspondence with both mean--field control problems and partial differential equations defined on the Wasserstein space.
 A comparison principle is established to ensure uniqueness, and existence results are obtained for generators that are linear or quadratic in the $z$--variable.
This framework provides a probabilistic approach to control and analysis on the space of probability measures.
\end{abstract}


\vspace{3mm}
\no{\bf MSC2010.} 60K35, 60H30, 91A13, 91A23, 91B30.

\section{Introduction}\label{sec:intro}

Backward Stochastic Differential Equations (BSDEs) have emerged as a powerful framework in the theory of stochastic processes, combining backward--in--time dynamics with adaptedness to a given filtration. Introduced rigorously in the nonlinear setting by \citeauthor*{pardoux1990adapted} \cite{pardoux1990adapted}, BSDEs have since become an important tool in the study of stochastic control problem and the theory of nonlinear partial differential equations (PDEs).

\medskip
In a filtered probability space $\left(\Omega, \mathcal{F}, \F:=(\mathcal{F}_t)_{t \in [0,T]}, \mathbb{P} \right)$ supporting a Brownian motion $(W_t)_{t \in [0,T]}$, an $\F$--adapted process $(Y_t,Z_t)_{t \in [0,T]}$ satisfying: for all $t \in [0,T]$,
\begin{align} \label{eq:intro_BSDE}
    Y_t = \xi + \int_t^T f_s( Y_s, Z_s) \, \mathrm{d}s - \int_t^T Z_s \, \mathrm{d}W_s,
\end{align}
is called a BSDE solution with given terminal condition
$\xi$ and generator $f$. Under square integrability of $\xi$ and Lipschitz assumptions on $f$, $(Y,Z)$ exits and is uniquely defined.  This initial framework has since been extensively generalized. Extensions include BSDEs with generators of quadratic or superquadratic growth (  \citeauthor*{kobylanski2000backward} \cite{kobylanski2000backward}), {\color{black}reflected BSDEs with obstacle constraints ( \citeauthor*{el1997reflected} \cite{el1997reflected} )}, BSDEs driven by L\'evy processes or jump terms, and second--order BSDEs (2BSDEs) that account for model uncertainty ( \citeauthor*{cheridito2007second} \cite{cheridito2007second}, \citeauthor*{soner2012wellposedness} \cite{soner2012wellposedness}). In addition, systems of BSDEs and multi--dimensional BSDEs have been studied in connection with game theory and risk--sharing problems.

\medskip
In stochastic control theory, BSDEs have usually been utilized through their connection with the Pontryagin maximum principle (see for instance \citeauthor*{bismutIntroductory1978} \cite{bismutIntroductory1978} and \citeauthor*{pengGeneral1990} \cite{pengGeneral1990}). These methods yield optimality conditions by characterizing adjoint processes and Hamiltonians. However, in this paper, we want to emphasize an alternative and powerful role of BSDEs--namely, their integration into the dynamic programming framework. Let us briefly explain this approach.

\begin{itemize}
    \item \textbf{From BSDE to stochastic control problem:} We consider the following stochastic control problem 
\begin{align*}
    V:=\sup_{\alpha \in \Ac} \E \left[\int_0^T L(t,X^\alpha_t,\alpha_t)\mathrm{d}t+g(X^\alpha_T) \right]\;\mbox{with}\;X^\alpha_\cdot=X_0 + \int_0^\cdot b(r,X^\alpha_r,\alpha_r)\mathrm{d}r +W_\cdot
\end{align*}
for some given maps $(b,L,g)$, and $\Ac$ the set of $A$--value $\F$--predictable processes $\alpha
$. The solution of the control problem is given in the following way. Let $(Y,Z)$ be a BSDE solution of \eqref{eq:intro_BSDE} when $f_s(y,z):=\sup_{a \in A}L(s,X_0+W_s,a) + z\;b(s,X_0+W_s,a)$ and $\xi:=g(X_0+W_T)$. Under mild assumptions, there exist Borel maps $\widehat{Y}:[0,T] \x \R \to \R$ and $\widehat{Z}:[0,T] \x \R \to \R$ s.t. $Y_t=\widehat{Y}(t,X_0+W_t)$ and $Z_t=\widehat{Z}(t,X_0+W_t)$ a.e. With $\widehat{\alpha}(s,w,z):=\arg \max_{a \in A}L(s,w,a) + z\;b(s,w,a)$, the solution is given by $\left(\widehat{\alpha} \left(t,X^\star_t,\widehat{Z}(t,X^\star_t) \right) \right)_{t \in [0,T]}$ where 
    \begin{align} \label{eq:intro_optimal}
        X^\star_\cdot=X_0 + \int_0^\cdot b\left(r,X^\star_r,\widehat{\alpha} \left(r,X^\star_r,\widehat{Z}(r,X^\star_r) \right) \right)\mathrm{d}r +W_\cdot
    \end{align}
    and $\E[Y_0]=V$. More generally, $Y_t$ corresponds to the dynamic version of $V$ that naturally satisfies a dynamic programming principle (see  \citeauthor*{el1997backward} \cite{el1997backward} for an overall and more details).

    \item \textbf{From stochastic control problem to BSDE} For any $\alpha$, let $\mathrm{d}\P^\alpha:=U^\alpha_T \mathrm{d}\P$ with  $\frac{\mathrm{d}U^\alpha_t}{U^\alpha_t}=b \left(t,X_0+W_t, \alpha_t \right)\mathrm{d}W_t$, $U^\alpha_0=1$. By setting, for all $t \in [0,T],$
    $$
        Y_t:= \esup_{\alpha \in \Ac} \E^{\P^\alpha} \left[ \int_t^T L(s,X_0 + W_s,\alpha_s)\mathrm{d}t+g(X_0 + W_T)\bigg| \Fc_t \right]=V(t,X_0+ W_t).
    $$
    If there exists an optimal control $\alpha^\star$. Then, there exists a process $Z$ s.t. $(Y,Z)$ is a BSDE solution of \eqref{eq:intro_BSDE} when $f_s(y,z):=\sup_{a \in A}L(s,X_0+W_s,a) + z\;b(s,X_0+W_s,a)$ and $\xi:=g(X_0+W_T)$.
\end{itemize}
 The use of BSDEs in this context aligns naturally with the value function approach and facilitates probabilistic representations of Hamilton--Jacobi--Bellman (HJB) equations. Compared to the maximum principle, in this non--degenerate setting, this formulation often provides greater flexibility, especially when dealing with non--smooth coefficients, path--dependence, or weak formulations of control.

\medskip
Another one of the most profound aspects of BSDE theory lies in its deep connection to partial differential equations (PDEs), particularly via the nonlinear Feynman–Kac formula. Let us provide some specifics
\begin{itemize}
    \item \textbf{From PDE to BSDE} Let $u:[0,T] \x \R$ be a smooth map satisfying: $u(T,\cdot)=g(\cdot)$ and for each $(t,x) \in [0,T) \x \R$,
    \begin{align} \label{eq:intro_PDE}
        \partial_t u(t,x) + H \left(t,x,u(t,x),\partial_x u(t,x) , \partial^2_xu(t,x) \right)=0\;\mbox{where}\;H(t,x,r,p,\gamma):=B(t,x)p+\frac{1}{2} \gamma + L(t,x,r,p).
    \end{align}
    Then $\left( u(t,X_t), \partial_x u(t,X_t) \right)_{t \in [0,T]}$ is a BSDE solution of \eqref{eq:intro_BSDE} when $f_t(y,z):=L(t,X_t,y,z)$, $\xi=g(X_T)$ and $\mathrm{d}X_t=B(t,X_t)\mathrm{d}t+\mathrm{d}W_t$.

    \item 
    \textbf{From BSDE to PDE} For any $(t_0,x_0)$, let $(Y^{t_0,x_0},Z^{t_0,x_0})$ be a BSDE solution when $f_t(y,z):=L(t,X^{t_0,x_0}_t,y,z)$, $\xi=g(X^{t_0,x_0}_T) $ and $X^{t_0,x_0}_{\cdot \vee t_0}=x_0 + \int_{t_0}^{\cdot \vee t_0} B(s, X^{t_0,x_0}_s)\mathrm{d}s + W_{\cdot \vee t_0}-W_{t_0}$. Then, under appropriate assumptions, the map $u(t,x):=Y^{t,x}_t$ is a viscosity solution of \eqref{eq:intro_PDE} (see \citeauthor*{pardoux1992backward} \cite{pardoux1992backward} and \citeauthor*{peng1991probabilistic} \cite{peng1991probabilistic}). 
\end{itemize}

\medskip
This connection between BSDEs and PDEs has led to fruitful developments in both theoretical and numerical aspects of control. For example, BSDE--based numerical schemes are now commonly used to approximate solutions of high--dimensional HJB equations, bypassing the curse of dimensionality typically associated with grid--based methods.

\medskip
These results emphasize the dual role of BSDEs, linking them both to semilinear PDEs via the nonlinear Feynman–Kac formula and to stochastic control problems through the dynamic programming principle. In this paper, we investigate how this correspondence extends to the mean--field setting. Let us mention what we mean by mean--field setting. The mean--field control (MFC) problem can be described by
\begin{align*}
    V:=\sup_{\alpha \in \Ac} \E \left[\int_0^T L(t,X^\alpha_t,\mu^\alpha_t,\alpha_t)\mathrm{d}t+g(X^\alpha_T) \right]\;\mbox{with}\;X^\alpha_\cdot=X_0 + \int_0^\cdot b(r,X^\alpha_r,\mu^\alpha_r,\alpha_r)\mathrm{d}r +W_\cdot\;\mbox{and}\;\mu^\alpha_r=\Lc(X^\alpha_r).
\end{align*}
and the PDE on the Wasserstein space is usually presented by looking for $u:[0,T] \x \Pc_2(\R)$ verifying $u(T,\cdot)=g(\cdot)$ and for $(t,m) \in [0,T) \x \Pc_2(\R)$,
\begin{align*}
    \partial_t u(t,m) + H \left(t,m,u,\partial_x \delta_m u , \partial^2_x \delta_mu \right)=0
\end{align*}
where $H (t,m, r,p,\gamma):=\int_\R p(t,m,x)\overline{B} \left(t, x, m \right) m(\mathrm{d}x) + \frac{1}{2} \int_\R \gamma(t,m,x) m(\mathrm{d}x) + \int_\R \overline{L} \left(t,x,m,r(t,m),p(t,m,x) \right) m(\mathrm{d}x),$ and $\delta_m u$ is the linear derivative w.r.t. to the measure (see the rigorous definition in \eqref{eq:def_derivative_measure}).  

\medskip
These problems have attracted significant interest in recent years due to the rich mathematical structure they present and their relevance in various applications, including economics, finance, and large population systems. For MFC problem, see \citeauthor*{andersson2011maximum} \cite{andersson2011maximum}, \citeauthor*{buckdahn2011general} \cite{buckdahn2011general}, \citeauthor*{bensoussan2013mean} \cite{bensoussan2013mean}, \citeauthor*{Carmona2013ForwardBackwardSD} \cite{Carmona2013ForwardBackwardSD}, \citeauthor*{pham2018bellman} \cite{pham2018bellman,pham2016dynamic}, \citeauthor*{lacker2017limit} \cite{lacker2017limit}, \citeauthor*{djete2019mckean} \cite{djete2019mckean,djete2019general} $\dots$.  See
\citeauthor*{Gangbo2008Hamilton} \cite{Gangbo2008Hamilton}, \citeauthor*{Buckdahn2014meanfield} \cite{Buckdahn2014meanfield}, \citeauthor*{Wu2018ViscosityST}  \cite{Wu2018ViscosityST}, \citeauthor*{CardaliaguetRegularity2023} \cite{CardaliaguetRegularity2023}, 
\citeauthor*{cosso2022masterbellmanequationwasserstein} \cite{cosso2022masterbellmanequationwasserstein}, \citeauthor*{meteViscosity2024} \cite{meteViscosity2024}, \citeauthor*{Bayraktar03042025} \cite{Bayraktar03042025}, \citeauthor*{Daudin01022025} \cite{Daudin01022025},   \citeauthor{bertucci2024stochasticoptimaltransporthamiltonjacobibellman}  \cite{bertucci2024stochasticoptimaltransporthamiltonjacobibellman}, \dots, for PDE over Wasserstein space.

\medskip
In this paper, we investigate the following central question: what is the appropriate extension of the classical BSDE pair $(Y,Z)$ that captures the dual relationship with both mean--field control problems and PDEs over the Wasserstein space ?

\medskip
We argue that the classical formulation of BSDEs is inadequate to capture the duality structure in the mean--field setting. A key limitation arises from the shift in the nature of the state variable: whereas in classical stochastic control the primary process is the controlled state trajectory $(X^\alpha_t)_{t \in [0,T]}$, it is now well understood that, in mean--field control problems, the central object is the law of the state process $(\mu^\alpha_t)_{t \in [0,T]}$. In the classical framework, the solution process $(Y_t)_{t \in [0,T]}$ of the BSDE represents the dynamic version of the value function, that is, $Y_t=V(t,X^{\alpha^\star}_t)$  along the optimal trajectory. By analogy, in the mean--field setting one would expect $Y_t=V(t,\mu^{\alpha^\star}_t)$  implying that $(Y_t)_{t \in [0,T]}$ is a deterministic process.

\medskip
However, this leads to a contradiction. If $Y$ is deterministic, then the martingale component $Z$ must vanish. Yet, as previously discussed ( see \eqref{eq:intro_optimal} ), the $Z$ process plays a critical role in the construction of the optimal control. Thus, a deterministic $Y$ process with $Z=0$ fails to capture the necessary information structure required for optimality.

\medskip
This inconsistency motivates the introduction of a new class of BSDEs, specifically designed to incorporate the distributional dependence inherent to mean--field control problems. Such a formulation aims to restore the correspondence between BSDEs, the value function, and the associated PDE on the Wasserstein space.

\medskip
In light of the observation that the primary object of interest is the probability distribution, a notion of a BSDE on the space of probability measures (the Wasserstein space) is introduced. In this framework, the process $Y$ is defined on  $[0,T] \x \Pc(\R)$. Furthermore, drawing on the concept of the linear functional derivative in the space of probability measures, the process $Z$ is defined on $[0,T] \x \R \x \Pc(\R)$. The precise definition is provided in \Cref{def:W-BSDE}.

\medskip
Under appropriate regularity and structural assumptions, we demonstrate that this generalized BSDE maintains a duality with the mean--field (McKean--Vlasov) control problem analogous to the well--established correspondence between classical BSDEs and standard stochastic control problems. Moreover, we establish a comparison principle for this new class of BSDEs, which in turn enables us to link their solutions to partial differential equations defined over the Wasserstein space, mirroring the classical connection between BSDEs and PDEs in finite--dimensional settings. 

\medskip
This perspective offers a unified approach to the study of mean--field control and PDEs over the Wasserstein space, potentially enriching the available tools for analyzing such problems. In the case of generators that are linear or quadratic in the $z$--variable, we are able to establish existence results for solutions of the proposed BSDEs, which in turn provide solutions to the associated mean--field control problems and PDEs on the space of probability measures.

\medskip
Although a general existence theory remains beyond the scope of this current work, this limitation is consistent with the known challenges in solving mean--field control problems and PDEs in the Wasserstein space. Nevertheless, the proposed formulation introduces a novel and coherent framework that opens new avenues for further research in this direction.

\medskip
An informed reader may notice the absence of common noise in our framework. This omission is intentional, motivated by two considerations. First, the no--common noise setting already poses significant mathematical challenges and has broad applicability. Second, the inclusion of common noise, with the techniques presently available to us, would necessitate a non--degenerate structure, thereby excluding the class of problems that arise uniquely in its absence i.e no--common noise case. We plan to investigate this direction in future work.



\medskip
The remainder of the paper is structured as follows. We begin by recalling some basic notations and introducing the probabilistic framework necessary to rigorously define the tools employed throughout the paper. In \Cref{sec:main}, we present our notion of BSDE on the Wasserstein space along with the main results. After providing the definition and the underlying intuition, \Cref{subsec:MCKean} explores the connection between our BSDE formulation and mean field control (MFC) problems. In \Cref{subsec:PDE}, we establish the duality with partial differential equations (PDEs) defined over the Wasserstein space. \Cref{subsec:comparison} is dedicated to a comparison principle and uniqueness result for the proposed BSDE. Finally, \Cref{subsec:existence} addresses the well--posedness of the equation under various conditions. All proofs are collected in \Cref{sec:proof}.


\medskip
{\bf \large Notations}.
	$(i)$
	Given a {\color{black}Polish} space $(E,\Delta)$ and $p \ge 1,$ we denote by $\Pc(E)$ the collection of all Borel probability measures on $E$,
	and by $\Pc_p(E)$ the subset of Borel probability measures $\mu$ 
	such that $\int_E \Delta(e, e_0)^p  \mu(de) < \infty$ for some $e_0 \in E$. The topology induced by the weak convergence will be used for $\Pc(E)$ and, for $p \ge 1$, we equip $\Pc_p(E)$ with the Wasserstein metric $\Wc_p$ defined by
	$
		\Wc_p(\mu , \mu') 
		~:=~
		\bigg(
			\inf_{\lambda \in \Lambda(\mu, \mu')}  \int_{E \x E} \Delta(e, e')^p ~\lambda( \mathrm{d}e, \mathrm{d}e') 
		\bigg)^{1/p},
	$
	where $\Lambda(\mu, \mu')$ denotes the collection of all probability measures $\lambda$ on $E \x E$ 
	such that $\lambda( \mathrm{d}e, E) = \mu$ and $\lambda(E,  \mathrm{d}e') = \mu'( \mathrm{d}e')$. Equipped with $\Wc_p,$ $\Pc_p(E)$ is a Polish space (see \cite[Theorem 6.18]{villani2008optimal}). 

\medskip
    \noindent $(ii)$ 
	Given a probability space $(\Om, \Hc, \P)$ supporting a sub--$\sigma$--algebra $\Gc \subset \Hc$ then for a Polish space $E$ and any random variable $\xi: \Om \longrightarrow E$, both the notations $\Lc^{\P}( \xi | \Gc)(\om)$ and $\P^{\Gc}_{\om} \circ (\xi)^{-1}$ are used to denote the conditional distribution of $\xi$ knowing $\Gc$ under $\P$.

\medskip

\medskip
    \noindent $(iii)$
	Let $\N^*$ denote the set of positive integers. Let $T > 0$ and $(\Sigma,\rho)$ be a Polish space, we denote by $C([0,T]; \Sigma)$ the space of all continuous functions on $[0,T]$ taking values in $\Sigma$.
	When $\Sigma=\R^k$ for some $k\in\N^*$, we simply write $\Cc^k := C([0,T]; \R^k)$. 

\medskip
    Given a map $U: \Pc(\R) \to \R$. we will say that a Borel map $F:\Pc(\R) \x \R \to \R$ is a linear functional derivative of $U$ if: for each $m$ and $m'$, we have $ \int_0^1 \int_\R \left|F\left(e\; m + (1-e)\; m',x\right) \right| \left(m+m' \right)(\mathrm{d}x)\; \mathrm{d}e< \infty$ and
    \begin{align} \label{eq:def_derivative_measure}
        U(m)-U(m')=\int_0^1 \int_\R F\left(e\; m + (1-e)\; m',x\right) \left(m-m' \right)(\mathrm{d}x) \mathrm{d}e.
    \end{align}
    We will denote $F$ by $\delta_m U$. Since $\delta_m U$ is defined up to a constant, we use the convention, $\int_\R \delta_m U(m,x)m(\mathrm{d}x)=0$ whenever $\int_\R |\delta_m U(m,x)|m(\mathrm{d}x) < \infty$.

\section{The notion of BSDE on the Wasserstein space and its applications} \label{sec:main}


Throughout this paper, we set a maturity time $T > 0$ and a probability space $(\Om,\H:=(\Hc_t)_{t \in [0,T]},\Hc,\P)$\footnote{ \label{footnote_enlarge}The probability space $(\Om,\H,\P)$ contains as many random variables as we want in the sense that: each time we need a sequence of independent uniform random variables or Brownian motions, we can find them in $\Om$ with the distribution $\P$ without mentioning an enlarging of the space. } supporting a real--valued $\H$--Brownian motion $(W_t)_{t \in [0,T] }$. We will say that $\mu:=(\mu_t)_{t \in [0,T]} \subset \Pc(\R)$ belongs to $\Lc$ if, for each $t$, $\mu_t=\Lc(X_t)$ for some process $(X_t)_{t \in [0,T]}$ satisfying $\mathrm{d}X_t=B(t,X_t) \mathrm{d}t + \mathrm{d}W_t$ where $B:[0,T] \x \R \to \R$ is a Borel bounded map. Let $f: [0,T] \x \R \x \Pc(\R) \x \R^2 \to \R$ be a Borel map that represents the generator, and $\psi:\Pc(\R) \to \R$ be a Borel map that plays the role of a terminal condition. The assumptions over $f$ and $\psi$ will be discussed later.

\paragraph*{Intuition of the defintion} 
Before introducing the formal definition, we provide some intuition based on an ``$n$--particle'' representation. Consider processes $(Y^n, Z^{1,n}, \dots, Z^{n,n})$ satisfying, for all $t \in [0,T]$,
\begin{align} \label{eq:n-BSDE-system}
    Y^n_t = \psi\left( \mu^n_T \right) + \frac{1}{n} \sum_{i=1}^n \int_t^T f\left(s, X^i_0 + W^i_s, \mu^n_s, Y^n_s, n Z^{i,n}_s\right) \, \mathrm{d}s - \frac{1}{n} \sum_{i=1}^n \int_t^T n Z^{i,n}_s \, \mathrm{d}W^i_s,
\end{align}
where $(W^i)_{i \ge 1}$ is a sequence of independent Brownian motions and the empirical measure $\mu^n_t$ is given by $\mu^n_t := \frac{1}{n} \sum_{i=1}^n \delta_{X^i_0 + W^i_t}.$ Under mild regularity assumptions, the existence and uniqueness of a solution to this system can be established using classical results from the BSDE literature. Given the structure of the system, it is natural to conjecture that the processes take the form $Y^n_t = Y(t, \mu^n_t) + \Oc(1/\sqrt{n})$ and $nZ^{i,n}_t = Z(t, X^i_0 + W^i_t, \mu^n_t) + \Oc(1/\sqrt{n})$ for some suitable maps $Y$ and $Z$. 
Let $\mu \in \Lc$ associated to a Borel bounded map $B:[0,T] \x \R \to \R$. We set $\mathrm{d}\P^n:=L^n_T \mathrm{d}\P$ where $\mathrm{d}L^n_t=L^n_t \sum_{i=1}^nB(t,X^i_0+W^i_t)\mathrm{d}W^i_t$ and $L^n_0=1$. The sequence $(X^i_0)_{i \ge 1 }$ is i.i.d with law $\mu_0$. By Girsanov Theorem, $\Lc(X) \otimes \cdots\otimes\Lc(X)=\Lc^{\P^n}(X_0^1+W^1,\cdots,X^n_0 + W^n)$ where $X$ is the process s.t. $\mu_t=\Lc(X_t)$. By going to the limit under $\P^n$ in \eqref{eq:n-BSDE-system}, we can obtain
\begin{align*}
    &Y(t,\mu_t)=\lim_{n \to \infty} \E^{\P^n}[Y^n_t]
    \\
    &=\lim_{n \to \infty}\E^{\P^n} \bigg[\psi\left( \mu^n_T \right) + \frac{1}{n} \sum_{i=1}^n \int_t^T f\left(s, X^i_0 + W^i_s, \mu^n_s, Y(s,\mu^n_s),  Z(s,X^i_0 + W^i_s,\mu^n_s)\right) \, \mathrm{d}s 
    \\
    &~~~~~~~~~~~~~~~~~~~~~~~~~~~~~~~~~~~~~~~~~~~~~~~~~~~~~~~~~~~~~~~~- \frac{1}{n} \sum_{i=1}^n \int_t^T  Z(s,X^i_0 + W^i_s,\mu^n_s) \, \mathrm{d}W^i_s \bigg]
    \\
    &=\psi(\mu_T) +  \E \left[ \int_t^T f \left( r, X_r, \mu_r, Y(r,\mu_r), Z \left(r,X_r, \mu_r \right) \right) \mathrm{d}r \right] - \E\left[ \int_t^T Z\left(r,X_r,\mu_r \right) B(r,X_r) \mathrm{d}r \right].
\end{align*}
The preceding discussion naturally leads to the following definition of a BSDE on the Wasserstein space.
We consider two Borel maps 
$$
    Y:[0,T] \x \Pc(\R) \to \R\;\mbox{and}\;Z:[0,T] \x \Pc(\R) \x \R \to \R.
$$

\begin{definition} \label{def:W-BSDE}
    We say that $(Y,Z)$ is a {\rm BSDE} solution on the Wasserstein space associated to $(f,\psi)$ if: for any $\mu \in \Lc$ associated to a drift $B$, the map $t \mapsto Y(t,\mu_t)$ is continuous,
    $$
        \E \left[ \int_0^T \left|f \left( r, X_r, \mu_r, Y(r,\mu_r), Z \left(r, X_r, \mu_r \right) \right) \right| + \left|Z\left(r, X_r,\mu_r \right) \right|\; \mathrm{d}r \right] < \infty,
    $$
    and for any $t \in [0,T]$,
    \begin{align} \label{eq:WBSDE_def}
        Y(t,\mu_t)
        =
        \psi(\mu_T) + \E \left[ \int_t^T f \left( r, X_r, \mu_r, Y(r,\mu_r), Z \left(r,X_r, \mu_r \right) \right) \mathrm{d}r \right] - \E\left[ \int_t^T Z\left(r,X_r,\mu_r \right)B(r,X_r)\; \mathrm{d}r \right].
    \end{align}
\end{definition}
To simplify notation, we refer to the pair 
$(Y,Z)$ as a solution to a Wasserstein BSDE (W--BSDE).

\begin{remark}
    {\rm (i)} The {\rm\Cref{eq:WBSDE_def}} must be verified by any family of probability measures $(\mu_t)_{t \in [0,T]}$ arising from a non--degenerate SDE i.e. $\Lc(X_t)=\mu_t$ with $\mathrm{d}X_t=B(t,X_t) \mathrm{d}t + \mathrm{d}W_t$ where $B:[0,T] \x \R \to \R$ is a Borel bounded map. The boundedness of $B$ can be relaxed and replaced by some linear growth condition as well as $\mathrm{d}W_t$ by $\sigma(t,X_t) \mathrm{d}W_t$ with $\sigma$ a non--degenrate diffusion coefficient. The non--degenracy aspect is important since we need to work with equivalent probability measures. Although our analysis is carried out in the one--dimensional case for clarity, the approach extends naturally to $\R^d$ with $d \in \{1,2,\cdots\}$. We adopt this simplified setting to improve readability and to emphasize the central ideas underlying this novel concept.

\medskip
{\rm (ii)} Although \eqref{eq:WBSDE_def} is a deterministic equation, the terminology ``{\rm BSDE}$"$ $($Backward Stochastic Differential Equation$)$ is retained for two reasons. First, it extends the classical notion of a {\rm BSDE} in a specific sense $($see {\rm \Cref{prop:linear}} for details$)$. Second, it preserves key structural constraints analogous to those in stochastic {\rm BSDEs}: the solutions $(Y, Z)$ must be adapted--i.e., at each time $t$, they depend only on $\mu_t$ and $X_t$—while ensuring the terminal condition $Y(T,\cdot) = \psi(\cdot)$ is met. This adaptedness property, combined with the requirement to hit a prescribed target at maturity, is central to the philosophy of classical {\rm BSDEs}.

\medskip
    {\rm (iii)} The joint consideration of $(Y, Z)$ in \eqref{eq:WBSDE_def} is essential, as the adaptedness constraint and the terminal condition cannot be decoupled. Neither $Y$ nor $Z$ alone suffices to encode these dynamics properly.

\medskip
{\rm (iv)} Observe that, under the square integrability condition \( \mathbb{E} \left[ \int_0^T |Z(r, X_r, \mu_r)|^2 \, \mathrm{d}r \right] < \infty \), it follows that the stochastic integral vanishes i.e. \( \mathbb{E} \left[ \int_0^T Z(r, X_r, \mu_r) \, \mathrm{d}W_r \right]=0 \). In this case, {\rm \Cref{eq:WBSDE_def}} can be rewritten in a more compact form:
\begin{align*}
    Y(t,\mu_t)
        =
        \psi(\mu_T) + \E \left[ \int_t^T f \left( r, X_r, \mu_r, Y(r,\mu_r), Z \left(r,X_r, \mu_r \right) \right) \mathrm{d}r \right] - \E\left[ \int_t^T Z\left(r,X_r,\mu_r \right) \mathrm{d}X_r \right].
\end{align*}
It is worth emphasizing that, in contrast to the classical case, the square integrability of \( Z \) is not a necessary condition for handling this formulation of {\rm BSDE}.


\medskip
    {\rm (v)} Finally, as with any differential equation, the initial time can be shifted: the equation remains well--posed when initialized at an arbitrary $t_0 \in [0,T]$ rather than at $0$ $($ see the definition in {\rm\Cref{eq:WBSDE_def_varrying_time}}$)$.

\end{remark}

\paragraph*{Coefficients with no--measure dependence} To gain initial insight into this definition, it is instructive to examine how it specializes in the classical setting, where the coefficients do not involve any dependence on a probability measure.  Let us consider the case where $\psi$ and $f$ are independent of $m$ i.e. $\left(\psi(m), f(t,x,m,y,z) \right)=\left(\langle \overline g, m \rangle, \overline f(t,x,z) \right)$ for some maps $x \mapsto \overline g(x)$ and $(t,x,z) \mapsto \overline f(t,x,z)$. We assume that $(\overline{g},\overline{f})$ is bounded and Lipschitz in $z$, uniformly in $(t,x)$. Let $(\overline Y, \overline{Z})$ be the $\left(\sigma \left( X_0, W_r:\; r \le t \right) \right)_{t \in [0,T]}$--adapted process satisfying: for all $t \le T$, a.e.
\begin{align*}
    \overline Y_t= \overline{g} \left(X_0+W_T  \right) + \int_t^T \overline{f} \left(r, X_0+W_r,Z_r \right) \mathrm{d}r - \int_t^T \overline{Z}_r \mathrm{d}W_r.
\end{align*}
Under the assumptions of $(\overline{g},\overline{f})$, the couple of processes $(\overline Y, \overline Z)$ are well defined, and they satisfy the integrability condition $\E \left[ \sup_{t \in [0,T]} |\overline{Y}_t|^2 + \int_0^T |\overline{Z}_t|^2 \mathrm{d}t \right] < \infty$ (see \cite{pardoux1990adapted}, and also \cite{JinRepresentation2002}). In addition, there exist Borel bounded map $u:[0,T] \x \R \to \R$ and Borel map  $v:[0,T] \x \R \to \R$ verifying $\overline{Y}_t=u(t,X_0+W_t)$ and $\overline{Z}_t=v(t,X_0+W_t)$ a.e.  We define
\begin{align*}
    \widehat{Y}(t,m)
    := \int_\R u(t,x) m(\mathrm{d}x)\;\mbox{and}\;\widehat{Z}(t,m,x):=v(t,x).
\end{align*}
\begin{proposition} \label{prop:linear}
    The couple $( \widehat{Y}, \widehat{Z})$ is a {\rm W--BSDE} solution in the sense of {\rm \Cref{def:W-BSDE}}.
\end{proposition}

\begin{proof}
    Let $\mu \in \Lc$ associated to $B$. We introduce the process $L$ verifying: $\mathrm{d}L_t= L_t B(t,X_0+W_t) \mathrm{d}W_t$ where $\Lc(X_0)=\mu_0$. Set $\mathrm{d}\P^{\mu}:= L_T \mathrm{d}\P$. By Girsanov Theorem, $W^B_\cdot:=W_\cdot-\int_0^\cdot B(s,X_0+W_s) \mathrm{d}s$ is a $\P^\mu$--Brownian motion. Notice that
    \begin{align*}
        \widehat{Y}(t,\mu_t)= \E\left[u(t, X_t) \right]=\E^{\P^\mu}[u(t,X_0+W_t)]=\E^{\P^\mu}\left[\overline{Y}_t\right].
    \end{align*}
    Since a.e., the map $t \mapsto \overline{Y}_t$  is continuous, we deduce that $t \mapsto \widehat{Y}(t,\mu_t)=\E^{\P^\mu}\left[\overline{Y}_t\right]$ is continuous. In addition,
    \begin{align*}
        \E \left[\int_0^T |v(t,X_t)| \mathrm{d}t \right] =\E^{\P^\mu} \left[\int_0^T |Z_t| \mathrm{d}t \right]= \E\left[L_T\int_0^T |Z_t| \mathrm{d}t \right] \le \E[TL_T^2]^{1/2} \E\left[ \int_0^T |Z_t|^2  \mathrm{d}t \right]^{1/2} < \infty
    \end{align*}
    and
    \begin{align*}
        \widehat{Y}(t,\mu_t)
        &=
        \E^{\P^\mu}\left[\overline{Y}_t\right]= \E^{\P^\mu}\left[\overline{g}(X_0+W_T) + \int_t^T \overline{f} \left(r, X_0+W_r,Z_r \right) \mathrm{d}r - \int_t^T \overline{Z}_r \mathrm{d}W_r \right]
        \\
        &=
        \E\left[\overline{g}(X_T) + \int_t^T \overline{f} \left(r, X_r,v(r,X_r) \right) \mathrm{d}r - \int_t^T v(r,X_r)B(r,X_r) \mathrm{d}r \right].
    \end{align*}
    This is enough to conclude the proof.
\end{proof}

The question of the well--posedness of the couple $(Y,Z)$ will be discussed in \Cref{subsec:existence} as well as the assumptions over $(f,\psi)$. Let us start by discussing the connection of our definition of {\rm W--BSDE} with the mean--field control problem and the PDEs over the Wasserstein space.

\subsection{Application to the optimal control of McKean--Vlasov equations} \label{subsec:MCKean}

In this section, we establish the dual relationship between our notion of W--BSDE and the mean--field control (MFC) problem. Specifically, we show that a solution to a W--BSDE can be used to construct a solution to the MFC problem, and conversely, that any solution to the MFC problem leads to a W--BSDE solution.

\medskip
Let us introduce the formulation of the optimal control of McKean--Vlasov equations used in the literature (see \cite{djete2019mckean}). Let $\Ac$ be a set of controls which are the set of Borel bounded maps $\alpha:[0,T] \x \R \to A$ for some compact set $A$. Let $(t,x,m,a) \mapsto b(t,x,m,a)$ and $(t,x,m,a) \mapsto L(t,x,m,a)$ be bounded Borel maps s.t. $b$ is Lipschitz in $(x,m)$ uniformly in $(t,a)$. We also assume that $\psi$ is bounded.
Let $t \in [0,T]$ be the initial time and $\nu \in \Pc(\R)$ be the initial distribution. For any $\alpha \in \Ac$, we introduce the process $X^{t,\nu,\alpha}:=X$ where: $\Lc(X_r)=\nu$ for each $r \in [0,t]$, $t \le s \le T$,
\begin{align*}
    \mathrm{d}X_s= b \left(s, X_s, \mu^\alpha_s, \alpha(s,X_s) \right) \mathrm{d}s + \mathrm{d}W_s\;\mbox{with}\;\mu^\alpha_s=\Lc(X_s).
\end{align*}
The process $X^{t,\nu,\alpha}$ is well defined (see, for instance, \cite{lacker2018OnStrong} or \cite[Theorem A.2.]{djete2019general}). The value function is
\begin{align*}
    V(t,\nu):= \sup_{\alpha \in \Ac} J(t,\nu,\alpha)\;\mbox{where}\;J(t,\nu,\alpha):=\E \left[ \int_t^T L \left(s, X^{t,\nu,\alpha}_s, \mu^\alpha_s, \alpha(s,X^{t,\nu,\alpha}_s) \right) \mathrm{d}s + \psi(\mu^\alpha_T) \right].
\end{align*}
We introduce $h(s,x,m,z,a):=L(s,x,m,a) + b(s,x,m,a) z$, $f(s,x,m,z):=\sup_{a \in A} h(s,x,m,z,a)$ and the \underline{unique} map $\widehat{a}$ verifying 
$$
    \widehat{a}(s,x,m,z):=\arg \max_{a \in A} h(s,x,m,z,a).
$$

\medskip
We consider $(Y,Z)$ a W--BSDE solution associated to $(f,\psi)$. Let $\mu \in \Lc$. We assume that there exists a process $R$ satisfying: $\Lc(R_r)=\mu_t$ for each $r \in [0,t]$, and $t \le s \le T$,
\begin{align} \label{eq:resolvant_McK}
     \mathrm{d}R_s= b \left(s, R_s, \nu_s, \widehat{a}(s,R_s,\nu_s, Z (s,R_s,\nu_s)) \right) \mathrm{d}s + \mathrm{d}W_s\;\mbox{with}\;\nu_s=\Lc(R_s).
\end{align}
We set $\widehat{\alpha}(s,x):=\widehat{a}(s,x,\nu_s, Z (s,x,\nu_s))$.

\begin{proposition} \label{prop:McK-fromBSDE}
    The control $\widehat{\alpha}$ is an optimal control for $V(t,\mu_t)$ and 
    $$
        Y(t,\mu_t)=V(t,\mu_t).
    $$
    In addition, for any other optimal control $\widetilde{\alpha}$, we have $\P$--a.e. for a.e. $r \in [t,T]$  
    $$
        \widetilde{\alpha}(r, X^{t,\mu_t,\tilde{\alpha}}_r)=\widehat{a}(r,X^{t,\mu_t,\tilde{\alpha}}_r,\mu^{\tilde \alpha}_r, Z (r,X^{t,\mu_t,\tilde{\alpha}}_r,\mu^{\tilde \alpha}_r)).
    $$
\end{proposition}

\begin{remark}
    {\rm (i) } The first part of {\rm \Cref{prop:McK-fromBSDE}} can be interpreted as a verification theorem in the spirit of classical stochastic control theory. It demonstrates how a solution to the mean--field control $($ MFC $)$ problem can be obtained via our notion of {\rm W--BSDE}. Moreover, it provides further insight into the function of the process $Z$, highlighting its direct involvement in the construction of the optimal control.

\medskip
    {\rm (ii)} The second part of {\rm \Cref{prop:McK-fromBSDE}} yields a necessary condition for optimality. However, this condition alone does not ensure uniqueness of the optimal control, due to the dependence of \( \widehat{a} \) on the distribution \( \mu^{\tilde \alpha} \) of the controlled process. Since distinct optimal trajectories may generate different marginal laws, uniqueness may fail. Nevertheless, if {\rm \eqref{eq:resolvant_McK}} admits a unique distributional solution, then uniqueness holds, as \( (\mu^{\tilde \alpha}_s)_{s \in [0,T]} \) and \( (\nu_s)_{s \in [0,T]} \) solve the same equation.


\end{remark}

Now, we will show that we can construct a W--BSDE solution through an optimal control of $V$. We denote by $\widehat{\alpha}:[0,T] \x \R \x \Pc(\R) \to A$ the optimal control for the problem $V$. For any $t \in [0,T]$, $\nu \in \Pc(\R)$, we assume that there exists $(R^{t,\nu}_s)_{s \in [0,T]}$ satisfying
$\Lc(R^{t,\nu}_r)=\nu$ for each $r \in [0,t]$, and $t \le s \le T$,
\begin{align*}
     \mathrm{d}R^{t,\nu}_s= b \left(s, R^{t,\nu}_s, \mu^{t,\nu}_s, \widehat{\alpha}(s,R^{t,\nu}_s,\mu^{t,\nu}_s) \right) \mathrm{d}s + \mathrm{d}W_s\;\mbox{with}\;\mu^{t,\nu}_s=\Lc(R^{t,\nu}_s).
\end{align*}
We set
\begin{align*}
    \widehat{Y}(t,\nu):=\E \left[ \int_t^T L \left(s, R^{t,\nu}_s, \mu^{t,\nu}_s, \widehat{\alpha}(s,R^{t,\nu}_s,\mu_s^{t,\nu}) \right) \mathrm{d}s + \psi\left(\mu^{t,\nu}_T \right) \right].
\end{align*}

We define the feedback--controlled drift and cost functions by
\[
\widehat{b}(t,x,m) := b\left(t,x,m, \widehat{\alpha}(t,x,m)\right), \quad \text{and} \quad \widehat{L}(t,m) := \int_{\mathbb{R}} L\left(t,x,m, \widehat{\alpha}(t,x,m)\right) m(\mathrm{d}x).
\]
We assume that the map $(t,x,m) \mapsto \big(\widehat{b}(t,x,m), \widehat{L}(t,m), \psi(m)\big)$ is continuous in $(x,m)$ for each fixed $t$, continuously differentiable in $x$ with bounded and continuous derivative, and admits a linear functional derivative with respect to the measure variable $m$, denoted by $\big(\delta_m \widehat{b}(t,x,m)(y), \delta_m \widehat{L}(t,m)(y), \delta_m \psi(m)(y)\big)$. Furthermore, these derivatives are bounded, continuous, continuously differentiable in $y$, and jointly continuous in $(x,m,y)$ for each $t$, with $\partial_y \delta_m \widehat{b}$, $\partial_y \delta_m \widehat{L}$, and $\partial_y \delta_m \psi$ also bounded.


\begin{proposition} \label{prop:McK-fromOC}
    We have $\widehat{Y}=V$ and there exists a map $\widehat{Z}:[0,T] \x \Pc(\R) \x \R \to \R$ s.t. the couple $\big(\widehat{Y}, \widehat{Z} \big)$ is a {\rm W--BSDE} solution associated to $(f,\psi)$ in the sense of {\rm \Cref{def:W-BSDE}}.
\end{proposition}

\begin{remark}
    {\rm \Cref{prop:McK-fromOC}} demonstrates that, under suitable additional assumptions, any solution to a mean--field control {\rm (MFC)} problem gives rise to a corresponding {\rm W--BSDE} solution. Taken together with {\rm \Cref{prop:McK-fromBSDE}}, this establishes a two--way correspondence between {\rm MFC} problems and {\rm W--BSDEs}: solving an {\rm MFC} problem yields a solution to a {\rm W--BSDE}, and, conversely, any solution to a {\rm W--BSDE} provides an optimal control for the associated {\rm MFC} problem.
\end{remark}

\subsection{Application to PDEs over the set of probability measures} \label{subsec:PDE}
We now investigate the fundamental connection between our Wasserstein BSDE {\rm (W-BSDE)} framework and partial differential equations (PDEs) on Wasserstein space. In this section, we establish a two--way correspondence:
\begin{itemize}
    \item Every classical solution of the associated Wasserstein PDE admits a representation as a solution to our W--BSDE;

    \item Conversely, any solution of the W--BSDE system can be interpreted as a viscosity solution of the corresponding PDE on the space of probability measures.
\end{itemize}


\medskip
Let $\overline{B}:[0,T] \x \R \x \Pc(\R) \to \R$ and $\overline{L}:[0,T] \x \R \x \Pc(\R) \x \R^2 \to \R$ be two bounded Borel maps. We introduce the generator
\begin{align*}
    H (t,m,r,p,\gamma):=\int_\R p(t,x,m)\overline{B} \left(t, x, m \right) m(\mathrm{d}x) + \frac{1}{2} \int_\R \gamma(t,m,x) m(\mathrm{d}x) + \int_\R \overline{L} \left(t,x,m, r(t,m),p(t,x,m) \right) m(\mathrm{d}x),
\end{align*}
for any $(t,m,p,\gamma)$ s.t. the previous integrals are well--defined. A function $v : [0,T] \times \Pc(\mathbb{R}) \to \mathbb{R}$ is said to be smooth if its time derivative $\partial_t v$ exists, and its linear functional derivative with respect to the measure variable, $\delta_m v$, is well--defined and twice differentiable with respect to the spatial variable. Moreover, all these derivatives are continuous and bounded in all their arguments.
Let $u$ be a smooth map satisfying: $u(T,\cdot)=\psi(\cdot)$ and for $(t,\nu) \in [0,T) \x \Pc(\R)$,
\begin{align}\label{eq:W-PDE}
    \partial_tu(t,\nu) + H\left(t,\nu,u,\partial_x \delta_mu, \partial^2_x \delta_mu \right)=0.
\end{align}
\begin{remark}
    Readers familiar with {\rm HJB} equations on the Wasserstein space may be more accustomed to working within $\Pc_2(\R)$ rather than $\Pc(\R)$. This preference is largely due to the Hilbertian structure provided by $\Pc_2(\R)$, which facilitates certain analytical techniques. In contrast, our approach does not rely on this structure, as we focus on settings with bounded coefficients. Nonetheless, most of the tools and results we develop can be extended to the $\Pc_2(\R)$ framework, provided the appropriate integrability conditions are satisfied.
\end{remark}

We define $f(t,x,m,y,z):=\overline{L}(t,x,m,y,z) + z \overline{B}(t,x,m)$ and, $\widehat{Y}(t,\nu)=u(t,\nu)$ and $\widehat{Z}(t,x,\nu):=\partial_x \delta_m u(t,x,\nu)$. 

\begin{proposition} \label{prop:PDE:fromPDE}
    The couple $(\widehat{Y},\widehat{Z} )$ is {\rm W--BSDE} solution associated to $(f,\psi)$ in the sense of {\rm \Cref{def:W-BSDE}}.
\end{proposition}

\begin{proof}
    It is just an application of It\^o Lemma\footnote{The It\^o formula along a flow of probability measures \( (\mu_t)_{t \in [0,T]} \) is usually formulated in \( \mathcal{P}_2(\mathbb{R}) \). However, since the functions involved in our framework are bounded, we can approximate any initial condition \( \mu_0 \in \mathcal{P}(\mathbb{R}) \) by a sequence in \( \mathcal{P}_2(\mathbb{R}) \), and thus extend the formula to \( \mathcal{P}(\mathbb{R}) \).
  }. Indeed, let $\mu \in \Lc$, since $u$ is solution of the PDE, we have $u(T,\mu_T)=\psi(\mu_T)$ and
    \begin{align*}
        &\mathrm{d}u(t,\mu_t)
        = \partial_t u(t,\mu_t) \mathrm{d}t + \E \left[ \partial_x \delta_m u(t,\mu_t)(X_t) B(t,X_t) \right] \mathrm{d}t + \frac{1}{2} \E \left[ \partial^2_x \delta_m u(t,\mu_t)(X_t) \right] \mathrm{d}t
        \\
        &= \partial_t u(t,\mu_t) \mathrm{d}t + H(t,\mu_t,u,\partial_x \delta_m u,\partial^2_x \delta_m u) \mathrm{d}t + \E \left[ \partial_x \delta_m u(t,\mu_t)(X_t) \left( B(t,X_t)- \overline{B}(t,X_t,\mu_t) \right) \right] \mathrm{d}t 
        \\
        &~~~~~~~~~~~~~~~~~~- \E \left[ \overline{L}\left(t,X_t,\mu_t,u(t,\mu_t), \partial_x \delta_m u(t,\mu_t)(X_t) \right) \right] \mathrm{d}t
        \\
        &=-\E \left[ f \left( t,X_t,\mu_t,u(t,\mu_t), \partial_x \delta_m u(t,\mu_t)(X_t) \right) \right] \mathrm{d}t +  \E \left[ \partial_x \delta_m u(t,\mu_t)(X_t) B(t,X_t) \right] \mathrm{d}t.
    \end{align*}
    As $\partial_x \delta_m u$ is considered bounded, we can then see that $(u,\partial_x \delta_m u)$  is a W--BSDE solution associated to $(f,\psi)$.
\end{proof}

Now, we connect a W--BSDE solution to the PDE. Let us mention the definition of viscosity solution associated to \eqref{eq:W-PDE}. We say that a continuous map $v:[0,T] \x \Pc(\R) \to \R$ is a viscosity sub--solution ( resp super--solution ) if: $v(T,\cdot) \le ({\rm resp} \ge ) \;\psi(\cdot)$ and, for any $(t,\nu)$ and smooth map $\varphi:[0,T] \x \Pc(\R) \to \R$ s.t. $(v-\varphi)$ has a maximum (resp minimum) at $(t,\nu)$ then
\begin{align*}
    \partial_tv(t,\nu) + H\left(t,\nu,v,\partial_x \delta_mv, \partial^2_x \delta_mv\right) \ge ({\rm resp} \le \;)\;0.
\end{align*}
We say that $v$ is a viscosity solution if it is both a viscosity sub--solution and a viscosity super--solution. Let $t_0 \in [0,T]$, we denote $(Y^{t_0},Z^{t_0})$ a W--BSDE solution associated to $(f,\psi)$ with initial time $t_0$ (see \Cref{eq:WBSDE_def_varrying_time} for the defintion). We set 
$$
    \widehat{u}(t_0,\nu_0):=Y^{t_0}(t_0,\nu_0).
$$
We assume that $\widehat{u}$ and $f$ are continuous, the map $(t,x,m,y,z) \mapsto f(t,x,m,y,z)$ is Lipscitz in $(y,z)$ uniformly in $(t,x,m)$ and, for each $t_0$, $Y^{t_0}$ and $Z^{t_0}$ are continuous in the measure argument $m$ for each $(t,x)$.
\begin{proposition} \label{prop:PDE:fromBSDE}
    The map $\widehat{u}:[0,T] \x \Pc(\R) \to \R$ is a viscosity solution of the {\rm PDE} \eqref{eq:W-PDE}. 
\end{proposition}

\begin{remark}
    In parallel with the mean--field control {\rm (MFC)} framework, {\rm \Cref{prop:PDE:fromPDE}} and {\rm \Cref{prop:PDE:fromBSDE}} highlight the correspondence between solutions of partial differential equations {\rm(PDEs)} over the Wasserstein space and solutions of {\rm W--BSDEs}. These results suggest that, in a meaningful sense, analyzing such {\rm PDEs} is equivalent to studying {\rm W--BSDEs}, thereby offering a dual perspective on the underlying problems.

\end{remark}

We now turn our attention to another desirable property for any notion of BSDE: namely, a comparison principle, which in turn implies a uniqueness result.


\subsection{Comparison and uniqueness results} \label{subsec:comparison}

We establish in this section a comparison principle for our notion of {\rm W--BSDE}. To this end, we consider {\rm W--BSDEs} with varying initial times. Let $t_0 \in [0,T]$. We say that a family of probability measures $(\mu_t)_{t \in [0,T]}$ belongs to the class $\Lc^{t_0}$ if there exists a bounded Borel function $B: [0,T] \times \R \to \R$ such that, for $s \in [t_0,T]$, $\mu_s = \Lc(R_s)$ where $(R_s)_{s \in [t_0,T]}$ solves
\[
\mathrm{d}R_s = B(s, R_s)\,\mathrm{d}s + \mathrm{d}W_s,
\]
and for $s \in [0,t_0]$, we set $\mu_s = \mu_{t_0}$. In other words, $\Lc^{t_0}$ consists of families of measures arising from non--degenerate diffusion processes starting from time $t_0$. In particular, we have $\Lc^0 = \Lc$. The couple $(Y^{t_0}, Z^{t_0})$ is a W--BSDE associated to $(f,\psi)$ starting at $t_0$ if: for any $\mu \in \Lc^{t_0}$ associated to $B$, the map $t \mapsto Y^{t_0}(t,\mu_t)$ is continuous, the integrability condition is satisfied i.e. $\E \left[ \int_{t_0}^T \left|f \left( r, X_r, \mu_r, Y^{t_0}(r,\mu_r), Z^{t_0} \left(r, X_r, \mu_r \right) \right) \right| + \left|Z^{t_0}\left(r, X_r,\mu_r \right) \right|\; \mathrm{d}r \right] < \infty,$
    and for any $t \in [t_0,T]$,
    \begin{align} \label{eq:WBSDE_def_varrying_time}
        Y^{t_0}(t,\mu_t)
        =
        \psi(\mu_T) + \E \left[ \int_t^T f \left( r, X_r, \mu_r, Y^{t_0}(r,\mu_r), Z^{t_0} \left(r,X_r, \mu_r \right) \right) \mathrm{d}r \right] - \E\left[ \int_t^T Z^{t_0}\left(r,X_r,\mu_r \right) B(r,X_r)\mathrm{d}r \right].
    \end{align}

\begin{remark}
    Considering different initial times proves useful for the following reason. Let $0 \le \underline{t_0} < \overline{t_0} \le T$. A measure flow $\mu = (\mu_t)_{t \in [0,T]} \in \Lc^{\overline{t_0}}$, i.e., arising from a non--degenerate diffusion on $[\overline{t_0},T]$, is not necessarily non--degenerate on the interval $[\underline{t_0},T]$. As a consequence, the relation in {\rm\Cref{eq:WBSDE_def_varrying_time}} may fail to hold for a solution pair $(Y^{\underline{t_0}}, Z^{\underline{t_0}})$ when $\mu \in \Lc^{\overline{t_0}}$. Thus, the ability to compare two {\rm W--BSDEs} with different initial times is essential, particularly for establishing the viscosity property in {\rm \Cref{prop:PDE:fromBSDE}}.
\end{remark}    


\medskip
Let $0 \le t_1 \le t_2 \le T$, $\left( f^i: [0,T] \x \R \x \Pc(\R) \x \R^2 \to \R \right)_{i=1,2}$ and $\left( \psi^i: \Pc(\R) \to \R \right)_{i=1,2}$ two bounded sequences of maps. We assume that: for each $i=1,2$, the map $(t,x,m,y,z) \mapsto f^i(t,x,m,y,z)$ is Lipschitz in $(y,z)$ uniformly in $(t,x,m)$, and continuous in $(m,y,z)$ for each $(t,x)$. We denote $\left( Y^{t_i}, Z^{t_i} \right)$ W--BSDE associated to $(f^i,\psi^i)$ for each $i=1,2$.
\begin{proposition} \label{prop:comparision}
    If $(t,x,m) \mapsto \left(Z^{t_i}(t,x,m), Y^{t_i}(t,m) \right)$ is continuous in $m$ for each $(t,x)$, for any $i=1,2$, 
\begin{align*}
    \int_\R f^1\left(t,x,\mu_t,Y^{t_2}(t,\mu_t), Z^{t_2}(t,x,\mu_t) \right)\mu_t(\mathrm{d}x) \le \int_\R f^2 \left(t,x,\mu_t,Y^{t_2}(t,\mu_t), Z^{t_2}(t,x,\mu_t) \right) \mu_t(\mathrm{d}x),\; \mathrm{d}t\mbox{--a.e. }
\end{align*}    
    $\psi^1 (\mu_T) \le \psi^2(\mu_T)$, for any $\mu \in \Lc^{t_2}$, then $Y^{t_1}(t,\mu_t) \le Y^{t_2}(t,\mu_t)$ for any $t \in [t_2,T]$ and $\mu \in \Lc^{t_1}$.
\end{proposition}

An immediate consequence of the previous proposition is the following uniqueness result. However, the statements concerning \( Z^{t_1} \) and \( Z^{t_2} \) are more delicate. The proof can be found in \Cref{proof:corr_uniqueness}.
\begin{corollary} \label{cor:uniqueness}
    Under the assumptions of the previous proposition, if 
\begin{align*}
    \int_\R f^1\left(t,x,\mu_t,Y^{t_2}(t,\mu_t), Z^{t_2}(t,x,\mu_t) \right) \mu_t(\mathrm{d}x) = \int_\R f^2 \left(t,x,\mu_t,Y^{t_2}(t,\mu_t), Z^{t_2}(t,x,\mu_t) \right) \mu_t(\mathrm{d}x),\;\mathrm{d}t\mbox{--a.e.} and 
\end{align*} 
    and $\psi^1 (\mu_T) = \psi^2(\mu_T)$,
    for any $\mu \in \Lc^{t_2}$ then $Y^{t_1}(t,\mu_t)=Y^{t_2}(t,\mu_t)$ for any $t \in [t_2,T]$ and $\mu \in \Lc^{t_1}$. In addition, if 
\begin{align*}
    f^1\left(t,x,\mu_t,Y^{t_2}(t,\mu_t), Z^{t_2}(t,x,\mu_t) \right) =  f^2 \left(t,x,\mu_t,Y^{t_2}(t,\mu_t), Z^{t_2}(t,x,\mu_t) \right),\;\mu_t(\mathrm{d}x)\mathrm{d}t\mbox{--a.e.}\;\mbox{for any }\mu \in \Lc^{t_2},
\end{align*}
    then for any $\mu \in \Lc^{t_1}$ associated to $B$, we have
    \begin{align*}
        Z^{t_1}(t,x,\mu_t)B(t,x)=Z^{t_2}(t,x,\mu_t)B(t,x),\;\mathrm{d}t\otimes\mathrm{d}x\mbox{--a.e.}(t,x)\in [t_2,T]\x\R.
    \end{align*}
\end{corollary}

\begin{remark}
    The comparison and uniqueness results presented here can be regarded as natural generalizations of the classical results for {\rm BSDEs}, adapted to the {\rm W--BSDE} framework. When \( Z^{t_1} \) and \( Z^{t_2} \) are independent of \( m \), the classical {\rm BSDE} equality is recovered. The presence of the measure argument, however, introduces additional complexity in establishing equality.



\end{remark}

\subsection{Some existence results} \label{subsec:existence}

This section is devoted to the well--posedness of {\rm W--BSDEs}. We begin with the case of vanishing generator, which parallels the classical martingale representation problem in standard {\rm BSDE} theory. We then address the case of a generator that is linear in the $z$ variable. While this case is typically handled via reduction to the zero--generator setting, additional technical challenges arise in our framework. We proceed by treating the case of a Lipschitz generator and conclude with the more delicate situation of a generator with quadratic growth in $z$.


\paragraph*{The case with no generator i.e. $f=0$} Let us consider $f=0$. This situation is equivalent to the martingale representation problem for the classical case. We assume that the map $\psi:\Pc(\R) \to \R$ is bounded and admits a bounded linear functional derivative $\delta_m \psi: \Pc(\R) \x \R \to \R$.

\medskip
We set 
$$
    Y(t,m):=\psi \left( \Lc\left( W_T-W_t + U \right) \right)
$$
and 
$$
    Z(t,x,m):={\E} \left[\delta_m \psi\left(  \Lc\left( W_T-W_t + U \right), W_T-W_t + x \right) \frac{W_T- W_t}{T-t} \right]
$$
where $U$ is a random variable independent of $W$ s.t. $\Lc\left( U\right)=m.$

\begin{proposition} \label{prop:Ex_no_generator}
    The couple $(Y,Z)$ is a {\rm W--BSDE} solution with generator $f=0$ and terminal map $\psi$. 
\end{proposition}

\begin{proof}
    This is essentially an application of \Cref{corollary:FK_measure_simplified} by taking $\overline{B}=0$.
\end{proof}

\begin{remark}
    In the classical {\rm BSDE} setting, where $\psi(m) = \langle g, m \rangle$ for some Borel bounded function $g:\R \to \R$, {\rm \Cref{prop:Ex_no_generator}} yields the expected result: no further assumptions on $g$ beyond boundedness and measurability are needed. In contrast, when $\psi$ exhibits nonlinearity in $m$, additional regularity is required—specifically, the existence of a linear functional derivative with respect to $m$. This requirement appears natural, as the problem is now set over the infinite--dimensional space $\Pc(\R)$ rather than the finite--dimensional space $\R$.

\end{remark}

\paragraph*{The case with linear generator} 
Now, let us assume that $f(t,x,m,y,z):=\ell(t,x,m)z$ where the map $\ell$ admits a bounded derivative in $x$ and, a bounded linear functional derivative $\delta_m \ell$ with bounded derivative $\partial_y \delta_m \ell$. Furthermore, the map $(t,x,m) \mapsto \ell(t,x,m)$ and its derivatives $(t,x,m) \mapsto \partial_x\ell(t,x,m)$, $(t,x,m,y) \mapsto\left( \delta_m \ell(t,x,m)(y), \partial_y\delta_m \ell(t,x,m)(y)  \right)$ are continuous in $(x,m,y)$ for each $t$.

\medskip
Let $(t,x,\nu)$ and an $\R$--valued $\Fc_t$--measurable random variable $\xi$ verifying $\Lc(\xi)=\nu$, we introduce the process $\left(X^{t,x,\nu}_s \right)_{s \in [t,T]}$ verifying: $X^{t,x,\nu}_t=x$ and for each $s \in [t,T]$,
\begin{align*}
    \mathrm{d}X^{t,x,\nu}_s= \ell \left(s, X^{t,x,\nu}_s, \nu_s \right) \mathrm{d}s + \mathrm{d}W_s\;\mbox{with}\;\nu_s=\Lc\left(X^{t,\xi,\nu}_s \right).
\end{align*}
We set 
$$
    Y(t,m):=\psi \left( \Lc\left( X^{t,U,m}_T \right) \right)
$$
where $U$ is a random variable independent of $W$ s.t. $\Lc\left( U\right)=m$. 

\begin{proposition} \label{prop:Ex_linear_gen}
    There exists a Borel map $Z:[0,T] \x \R \x \Pc(\R) \to \R$ s.t. the couple $(Y,Z)$ is a {\rm W--BSDE} solution with generator $f$ and terminal map $\psi$. 
\end{proposition}

\begin{proof}
    This is an application of \Cref{prop:FK_measure_general} by taking $\overline{B}(t,x,m)=\ell(t,x,m)$.
\end{proof}

\begin{remark}
    In line with the {\rm PDEs} perspective on {\rm W--BSDEs}, {\rm \Cref{prop:Ex_linear_gen}} can be interpreted as an analogue of the classical {\rm Feynman--Kac} representation for linear {\rm PDEs} defined over the Wasserstein space. Similar representation results have been established in the literature for linear equations on the Wasserstein space $($see, for instance, {\rm \cite{Buckdahn2014meanfield}}$)$. However, our result requires fewer assumptions, particularly regarding the regularity in the measure variable for the coefficients.

\end{remark}

\subsubsection{The case with Lipschitz generator} 

\paragraph*{Generator Lipschitz 
in $y$ linear in $z$}

We consider the case where the generator is given by
\[
f(t,x,m,y,z) := h(t,m,y) + P(t,x)z,
\]
where the function $h$ admits a linear functional derivative in the measure variable, denoted by $\delta_m h(t,m,y)(e)$. We assume that the map
\[
[0,T] \times \R \times \Pc(\R) \x \R \x \R \ni (t,x,m,y,e) \mapsto \left(h(t,m,y),\, \delta_m h(t,m,y)(e),\, P(t,x)\right) \in \R^3
\]
is Borel measurable and bounded, and is Lipschitz continuous in $(x,y)$ uniformly with respect to $(t,m)$.


\begin{proposition} \label{prop:existY}
    There exists a couple $(Y,Z)$ a {\rm W--BSDE} solution with generator $f$ and terminal map $\psi$.
\end{proposition}

\paragraph*{General Lipschitz} We are not able to provide a general existence result for generator Lipschitz in $z$. However, we can give a condition that leads directly to a solution. 

\medskip
Given Borel maps $y:[0,T] \x \Pc(\R)\to \R$ and $z:[0,T] \x \R \x \Pc(\R) \to \R$ s.t. for each $(t,m)$, we can define
\begin{align*}
    {K}^{y,z}(t,m):=\int_{\R} f(t,x,m,y(t,m),z(t,x,m))m(\mathrm{d}x),
\end{align*}
and for a random variable $U$ independent of $W$ s.t. $\Lc(U)=m$, 
\begin{align*} 
    \overline{K}^{y,z}(t,m):=\int_t^T {K}^{y,z}\left(s,\Lc(W_s-W_t+U) \right) \;\mathrm{d}s + \psi(\Lc(W_T-W_t+U))
\end{align*}
and 
\begin{align*} 
    \delta_m\overline{K}^{y,z}(t,m)(x)&:= \E \left[\int_t^T   \delta_m {K}^{y,z}\left(s, \Lc(W_s-W_t+U) \right) (W_s-W_t + x)  \frac{W_s-W_t}{s-t} \mathrm{d}s \right] 
    \\
    &~~~~~~+ \E \left[ \delta_m \psi (\Lc(W_T-W_t+U))(W_T-W_t + x) \frac{W_T-W_t}{T-t}  \right].
\end{align*}

\begin{proposition} \label{pro:reduced_existence}
    If there exist Borel maps $Y:[0,T] \x \Pc(\R) \to \R$ and $Z:[0,T] \x \R \x \Pc(\R) \to \R$ s.t. for any $\mu \in \Lc$, the map $t \mapsto Y(t,\mu_t)$ is continuous, $\E \left[ \int_0^T \left|f \left( r, X_r, \mu_r, Y(r,\mu_r), Z \left(r, X_r, \mu_r \right) \right) \right| + \left|Z\left(r, X_r,\mu_r \right) \right|\; \mathrm{d}r \right] < \infty$ and for $(t,x,m)$,
    \begin{align} \label{eq:reduced}
        Y(t,m)=\overline{K}^{Y,Z}(t,m)\;\mbox{and}\;Z(t,x,m)=\delta_m\overline{K}^{Y,Z}(t,m)(x)
    \end{align}
    then $(Y,Z)$ is a {\rm W--BSDE} solution associated with $(f,\psi)$.

\end{proposition}

\begin{proof}
    It is clear that $Y(T,\mu_T)=\psi(\mu_T)$. We just need to compute the dynamics of $Y(t,\mu_t)$. By an application of \Cref{corollary:FK_measure_simplified}, we get, for $s \le t$,
    \begin{align*}
         &K^{Y,Z} \left(s, \Lc\left( W_s-W_t + U_t \right) \right)
         \\&=K^{Y,Z}\left(s,\Lc(U_s) \right) 
         + \E \left[ \int_t^s \delta_m K^{Y,Z}\left(s, \Lc\left( W_s-W_r + U_r \right)  \right) \left( W_s-W_r + U_r \right) \frac{W_s-W_r}{s-r} \left( 0 - B(r,U_r) \right) \mathrm{d}r \right].
    \end{align*}
    Therefore, 
    \begin{align*}
        \mathrm{d}Y(t,\mu_t)= -\E \left[ f(t,X_t,\mu_t, Y(t,\mu_t), Z(t,X_t,\mu_t)) \right]\mathrm{d}t + \E \left[ Z(t,X_t,\mu_t) B(t,X_t) \mathrm{d}t \right].
    \end{align*}
    
\end{proof}

\begin{remark}
    {\rm \Cref{pro:reduced_existence}} shows that the existence of solutions to the {\rm W--BSDE} can be reduced to solving a pair of equations, namely {\rm \eqref{eq:reduced}}. Moreover, it can be shown that any solution to a {\rm W--BSDE} satisfies {\rm \eqref{eq:reduced}} in a weaker sense. This observation highlights the relevance of understanding {\rm \eqref{eq:reduced}} in the broader study of {\rm W--BSDEs}.

\medskip
However, the system {\rm \eqref{eq:reduced}} presents notable difficulties. Unlike standard fixed--point formulations, it does not lend itself easily to classical methods. The equation for $Z$ involves $Z$ itself on the left--hand side, while the right--hand side potentially includes its linear functional derivative with respect to the measure, $\delta_m Z$. This creates an asymmetry: while $Z$ may be expressed as a function of $\delta_m Z$, we cannot directly express $\delta_m Z$ as a function of $Z$. This complicates the construction of a solution, particularly through fixed--point techniques.

\medskip
We therefore leave this challenging and intriguing question for future research. It is worth emphasizing, however, that the difficulty involving $Z$ and $\delta_m Z$ in {\rm \eqref{eq:reduced}} may still be milder than the one encountered in seeking classical solutions of {\rm PDEs} over the Wasserstein space, since {\rm \eqref{eq:reduced}} is linear in $\delta_m Z$. As such, the {\rm W--BSDE} framework might offer a more tractable alternative for studying {\rm PDEs} on the space of probability measures.
\end{remark}

\subsubsection{The case of quadratic generator} 

It is well known in the classical theory of {\rm BSDEs} that quadratic generators can be treated under suitable conditions. Although we are not yet able to address the case of general Lipschitz generators in the {\rm W--BSDE} framework, we are able to handle the specific case of a quadratic generator, namely $f(z) := \frac{1}{2} z^2$. This is made possible by the fact that an explicit solution to the corresponding {\rm W--BSDE} can be constructed in this setting.


\medskip
Given two probability measures $m$ and $m'$, we write $m * m'$ the measure defined by $m * m':=\Lc(X+X')$ where $X$ and $X'$ are independent, and $\Lc(X)=m$ and $\Lc(X')=m'$. We introduce the relative entropy $D$ by \begin{align*}
        D(m|\mu):=\int_{\R} \frac{\mathrm{d}m}{\mathrm{d}\mu} \log \left( \frac{\mathrm{d}m}{\mathrm{d}\mu}  \right) \mathrm{d}\mu\mbox{ if }m \ll\mu,\;D(m|\mu)=\infty\mbox{ otherwise}.
    \end{align*}

    \medskip
    We assume that \( \psi \) is continuous bounded, concave, and admits a continuous, bounded linear functional derivative \( \delta_m \psi(m)(x) \), which is differentiable in \( x \) with continuous and bounded derivative \( \partial_x \delta_m \psi \).

\begin{proposition} \label{prop:existSquare}
      There exists $(Y,Z)$ a {\rm W--BSDE} solution with quadratic generator and terminal map $\psi$ satisfying:
    \begin{align*}
        Y(t,m)
        =
        \sup_{m' \in \Pc(\R)} \psi(m'*m) -D\left(m'\;|\;\Lc(W_{T-t}) \right)
    \end{align*}
    and
    \begin{align*}
        Z(t,x,m)
        =
         \int_{\R}  \partial_x \delta_m {\psi}\left( \nu(t,m) * m \right) \left(z + x \right)
        \nu(t,m)(\mathrm{d}z)
    \end{align*}
    where $\nu(t,m)$ is verifying:
    \begin{align*}
        \nu(t,m)(\mathrm{d}z)
        =
        \frac{\exp\left\{ \int_{\R}  \delta_m {\psi}(\nu(t,m)*m)\left( z+x\right) m(\mathrm{d}x) \right\} }{\int_{\R} \exp\left\{ \int_\R  \delta_m {\psi}(\nu(t,m)*m)\left( z'+ x\right) m(\mathrm{d}x)\right\} \Lc\left(W^1_{T-t}\right)(\mathrm{d}z') } \Lc\left(W^1_{T-t}\right)(\mathrm{d}z).
    \end{align*}
\end{proposition}
\begin{remark}
    {\rm (i) }A nonlinear dependence in the variable $z$ is addressed in {\rm \Cref{prop:existSquare}}. In the context of the {\rm (MFC)} problem, this result can be interpreted—via a straightforward extension of {\rm \Cref{prop:McK-fromBSDE}} to the unbounded setting—as providing an explicit characterization of the optimal control when the data takes the form $L(t,x,m,a) = -\frac{1}{2}a^2$ and $b(t,x,m,a) = a$. To the best of our knowledge, this type of explicit solution for the {\rm MFC} problem has not been previously obtained under the relatively mild conditions we impose on $\psi$, which are weaker than those typically assumed in the existing literature relying on the Pontryagin principle.

\medskip
    {\rm (ii) } 
    By following the same approach as in the proof of {\rm \Cref{prop:existSquare}} $($see \Cref{sec:proof:square}$)$, one can directly extend the comparison principle established in {\rm \Cref{prop:comparision}} to the quadratic case. This, in turn, implies the existence of a viscosity solution to the corresponding {\rm PDE} on the Wasserstein space.

\medskip
    {\rm (iii)} Observe that the distribution $\nu(t,m)$ identified in {\rm \Cref{prop:existSquare}} satisfies a nonlinear Fokker--Planck-type equation and admits an explicit representation as a Gibbs measure. This structure is significant, as it reveals a deep connection between the formulation of {\rm W--BSDEs} and equilibrium distributions commonly encountered in statistical mechanics. In particular, the Gibbs property encodes optimality and entropy--regularized control features, offering further insight into the analytical behavior and regularity of solutions to {\rm W--BSDEs}.


    

\end{remark}

\section{Proofs of the main results}  \label{sec:proof}

\subsection{Optimal control of McKean--Vlasov equations}

\subsubsection{Proof of \Cref{prop:McK-fromBSDE}}

\begin{proof}
    Let $\alpha \in \Ac$ and $\mu \in \Lc$. We recall that $\mu_s=\Lc(S_s)$ where $\mathrm{d}S_s=B(s, S_s)\mathrm{d}s + \mathrm{d}W_s$ where $B$ is a bounded Borel map. Let $\overline{b}(s,x,m,a):=B(s,x) \1_{s \le t} + b(t,x,m,a) \1_{s > t} $. Let $(U)_{s \in [0,T]}$ be the process satisfying $\mathrm{d}U_s= \overline{b}(s,U_s,\Lc(U_s), \alpha(s,U_s))\mathrm{d}s + \mathrm{d}W_s$ with $\Lc(U_0)=\mu_0$. The process $U$ is well--defined in distribution by \cite[Theorem 2.4.]{lacker2018OnStrong}. Since $\overline{b}$ is bounded, we can see that $\left( \eta_s:=\Lc(U_s) \right)_{s \in [0,T]}$ belongs to $\Lc$, and we can check that $\Lc(X^{t,\mu_t,\alpha}_s)=\Lc(U_s)=\eta_s$ for any $s \in [t,T]$. Therefore, as $(Y,Z)$ is a W--BSDE solution,
    \begin{align*}
        &Y(t,\mu_t)=Y(t,\eta_t)
        \\
        &=
        \psi(\eta_T) + \E \left[ \int_t^T f \left( r, U_r, \eta_r, Y(r,\eta_r), Z \left(r,U_r, \eta_r \right) \right) \mathrm{d}r \right] - \E\left[ \int_t^T Z\left(r,U_r,\eta_r \right) b\left(r,U_r,\eta_r,\alpha(r,U_r) \right) \mathrm{d}r \right]
        \\
        &=
        \psi(\mu^\alpha_T) + \E \left[ \int_t^T f \left( r, X^{t,\mu_t,\alpha}_r, \mu^\alpha_r, Y(r,\mu^\alpha_r), Z \left(r,X^{t,\mu_t,\alpha}_r, \mu^\alpha_r \right) \right) \mathrm{d}r \right] 
        \\
        &~~~~~~~~~~~~- \E\left[ \int_t^T Z\left(r,X^{t,\mu_t,\alpha}_r,\mu^\alpha_r \right) b\left( r,X^{t,\mu_t,\alpha}_r,\mu^\alpha_r, \alpha(r,X^{t,\mu_t,\alpha}_r) \right)\mathrm{d}r \right].
    \end{align*}
    Consequently,
    \begin{align*}
        &J(t,\mu_t,\alpha)=\E \left[ \int_t^T L \left(s, X^{t,\mu_t,\alpha}_s, \mu^\alpha_s, \alpha(s,X^{t,\mu_t,\alpha}_s) \right) \mathrm{d}s + \psi(\mu^\alpha_T) \right]
        \\
        &= \E \left[ \int_t^T L \left(s, X^{t,\mu_t,\alpha}_s, \mu^\alpha_s, \alpha(s,X^{t,\mu_t,\alpha}_s) \right) \mathrm{d}s \right] + Y(t,\mu_t) 
        \\
        &~~~~~
        - \E \left[ \int_t^T f \left( r, X^{t,\mu_t,\alpha}_r, \mu^\alpha_r, Z \left(r,X^{t,\mu_t,\alpha}_r , \mu^\alpha_r\right) \right) \mathrm{d}r \right] 
        + \E\left[ \int_t^T Z\left(r,X^{t,\mu_t,\alpha}_r,\mu^\alpha_r \right) b\left( r,X^{t,\mu_t,\alpha}_r,\mu^\alpha_r, \alpha(r,X^{t,\mu_t,\alpha}_r) \right)\mathrm{d}r \right]
        \\
        &=\E \left[ \int_t^T h \left( r, X^{t,\mu_t,\alpha}_r, \mu^\alpha_r, Z \left(r,X^{t,\mu_t,\alpha}_r, \mu^\alpha_r \right), \alpha(r, X^{t,\mu_t,\alpha}_r) \right) - f \left( r, X^{t,\mu_t,\alpha}_r, \mu^\alpha_r, Z \left(r,X^{t,\mu_t,\alpha}_r , \mu^\alpha_r\right) \right) \mathrm{d}r + Y(t,\mu_t)  \right],
    \end{align*}
    by the definition of $f$, $J(t,\mu_t,\alpha) \le Y(t,\mu_t)$.
    This is true for any $\alpha \in \Ac$. Then, $V(t,\mu_t) \le Y(t,\mu_t)$. By taking $\widehat{\alpha}$, we obtain the equality and the previous inequality. Consequently, $\widehat{\alpha}$ is an optimal control and $V(t,\mu_t)=Y(t,\mu_t)$. For any other optimal control $\widetilde{\alpha}$, we find
    \begin{align*}
        &V(t,\mu_t)=J(t,\mu_t,\widetilde{\alpha})
        \\
        &= \E \left[ \int_t^T h \left( r, X^{t,\mu_t,\tilde{\alpha}}_r, \mu^{\tilde{\alpha}}_r, Z \left(r,X^{t,\mu_t,\tilde{\alpha}}_r, \mu^{\tilde{\alpha}}_r \right),\tilde{\alpha}(r, X^{t,\mu_t,\tilde{\alpha}}_r) \right) - f \left( r, X^{t,\mu_t,\tilde{\alpha}}_r, \mu^{\tilde{\alpha}}_r, Z \left(r,X^{t,\mu_t,\tilde{\alpha}}_r, \mu^{\tilde{\alpha}}_r \right) \right) \mathrm{d}r + Y(t,\mu_t)  \right].
    \end{align*}
    Consequently,
    \begin{align*}
        \E \left[ \int_t^T h \left( r, X^{t,\mu_t,\tilde{\alpha}}_r, \mu^{\tilde{\alpha}}_r, Z \left(r,X^{t,\mu_t,\tilde{\alpha}}_r, \mu^{\tilde{\alpha}}_r \right),\tilde{\alpha}(r, X^{t,\mu_t,\tilde{\alpha}}_r) \right) - f \left( r, X^{t,\mu_t,\tilde{\alpha}}_r, \mu^{\tilde{\alpha}}_r, Z \left(r,X^{t,\mu_t,\tilde{\alpha}}_r, \mu^{\tilde{\alpha}}_r \right) \right) \mathrm{d}r  \right]=0.
    \end{align*}
    Since $f(t,x,m,z) \ge h(t,x,m,z,a)$ for any $(t,x,m,z,a)$, we have 
    $$
        h \left( r, X^{t,\mu_t,\tilde{\alpha}}_r, \mu^{\tilde{\alpha}}_r, Z \left(r,X^{t,\mu_t,\tilde{\alpha}}_r , \mu^{\tilde{\alpha}}_r\right),\tilde{\alpha}(r, X^{t,\mu_t,\tilde{\alpha}}_r) \right) - f \left( r, X^{t,\mu_t,\tilde{\alpha}}_r, \mu^{\tilde{\alpha}}_r, Z \left(r,X^{t,\mu_t,\tilde{\alpha}}_r, \mu^{\tilde{\alpha}}_r \right) \right) =0\;\mbox{a.e.}
    $$
    By uniqueness of $\widehat{a}$, $\widetilde{\alpha}(r, X^{t,\mu_t,\tilde{\alpha}}_r)=\widehat{a}\left(r,X^{t,\mu_t,\tilde{\alpha}}_r,\mu^{\tilde \alpha}_r, Z (r,X^{t,\mu_t,\tilde{\alpha}}_r,\mu^{\tilde \alpha}_r) \right)$ a.e. 
\end{proof}

\subsubsection{Proof of \Cref{prop:McK-fromOC}}

\begin{proof}
    Since $\widehat{\alpha}$ is optimal for $V$, we have $\widehat{Y}(t,\nu)=V(t,\nu)$ for any $(t,\nu)$. Let $\mu \in \Lc$ associated to $B$. Under the assumptions verified by $(\widehat{b},\widehat{L},\psi)$, we apply \Cref{prop:FK_measure_general} and find Borel bounded maps $\left(Z^\psi, Z^L \right):[0,T] \x \R \x \Pc(\R) \to \R$ s.t  
    \begin{align*}
        \mathrm{d} \widehat{Y} (t,\mu_t)
        =
        &-\E \left[ L \left(t, X^{t,\mu_t,\hat{\alpha}}_t, \mu^{\mu_t,\hat{\alpha}}_t, \widehat{\alpha}\left(t,X^{t,\mu_t,\hat{\alpha}}_t, \mu_t^{\mu_t, \hat \alpha}\right) \right) \right] \mathrm{d}t 
        \\
        &+   \int_t^T \E \left[ Z^L(s,X^{t,\mu_t, \hat \alpha}_s,\mu^{\mu_t,\hat \alpha}_s)  \left( B\left(t, X^{t,\mu_t,\hat{\alpha}}_t\right) - \widehat{b}\left(t, X^{t,\mu_t,\hat{\alpha}}_t,\mu_t \right) \right) \right] \mathrm{d}s\; \mathrm{d}t
        \\
        &+   \E \left[Z^\psi\left(t, X^{t,\mu_t,\hat{\alpha}}_t,\mu_t \right) \left( B\left(t, X^{t,\mu_t,\hat{\alpha}}_t\right) - \widehat{b}\left(t, X^{t,\mu_t,\hat{\alpha}}_t, \mu_t \right) \right) \right] \mathrm{d}t.
    \end{align*}

    We define
    \begin{align*}
        \widehat{Z}(t,x,\nu)
        &:=
        \int_t^T \E \left[ Z^L(s,X^{t,\nu, \hat \alpha}_s,\mu^{\nu,\hat \alpha}_s)  \bigg| X^{t,\nu,\hat{\alpha}}_t=x  \right] \mathrm{d}s + Z^{\psi}(t,x,\nu) . 
    \end{align*}
    Therefore,
    \begin{align*}
        \widehat{Y} (t,\mu_t)
        =
        \psi(\mu_T) + \E \left[\int_t^T L \left(s, X_s, \mu_s, \widehat{\alpha}(s,X_s,\mu_s) \right) + \widehat{Z}(s,X_s,\mu_s) \widehat{b}(s,X_s,\mu_s) \mathrm{d}s \right] - \E \left[ \int_t^T \widehat{Z}(s,X_s,\mu_s) B(s,X_s) \mathrm{d}s \right].
    \end{align*}
    To conclude, we need to show that $L \left(s, X_s, \mu_s, \widehat{\alpha}(s,X_s,\mu_s) \right) + \widehat{Z}(s,X_s,\mu_s) \widehat{b}(s,X_s,\mu_s)=f\left(s,X_s,\mu_s, \widehat{Z}(s,X_s,\mu_s) \right)$ $\mathrm{d}t \otimes\mathrm{d}\P$ a.e..
    Let $\mu \in \Lc$. We observe that for any $\alpha \in  \Ac$ and $t_0 \in [t,T]$,  we have
    \begin{align*}
        &
        \E \left[ \int_t^{t_0} L \left(s, X^{t,\mu_t,\alpha}_s, \mu^\alpha_s, \alpha(s,X^{t,\mu_t,\alpha}_s) \right) \mathrm{d}s + Y(t_0,\mu^\alpha_{t_0}) \right]
        \\
        &
        =\widehat{Y}(t,\mu_t) + \E \bigg[ \int_t^{t_0} h \left( s, X^{t,\mu_t,\alpha}_s, \mu^\alpha_s, \widehat{Z}\left(s,X^{t,\mu_t,\alpha}_s,\mu^\alpha_s\right), \alpha(s,X^{t,\mu_t,\alpha}_s) \right) 
        \\
        &~~~~~~~~~~~~~~~~~~~~~~~~~~~~~~- h \left( s, X^{t,\mu_t,\alpha}_s, \mu^\alpha_s, \widehat{Z}\left(s,X^{t,\mu_t,\alpha}_s,\mu^\alpha_s\right), \widehat{\alpha}(s,X^{t,\mu_t,\alpha}_s,\mu^\alpha_s) \right) \mathrm{d}s \bigg].
    \end{align*}
    By using the fact that $\widehat{Y}(t,\mu_t)$ is the value function combining with Dynamic Programming Principle (see \cite[Theorem 3.4.]{djete2019mckean}), this leads to: for any $(t,\alpha)$ and $t _0 \in [t,T]$, $c_{t}(t_0) \le 0$ where
    \begin{align*}
        &c_{t}(t_0):=\E \left[ \int_t^{t_0} L \left(s, X^{t,\mu_t,\alpha}_s, \mu^\alpha_s, \alpha(s,X^{t,\mu_t,\alpha}_s) \right) \mathrm{d}s + Y(t_0,\mu^\alpha_{t_0}) \right] - \widehat{Y}(t,\mu_t)
        \\
        &=\E \left[ \int_t^{t_0} h \left( s, X^{t,\mu_t,\alpha}_s, \mu^\alpha_s, \widehat{Z}\left(s,X^{t,\mu_t,\alpha}_s,\mu^\alpha_s\right), \alpha(s,X^{t,\mu_t,\alpha}_s) \right) - h \left( s, X^{t,\mu_t,\alpha}_s, \mu^\alpha_s, \widehat{Z}\left(s,X^{t,\mu_t,\alpha}_s,\mu^\alpha_s\right), \widehat{\alpha}(s,X^{t,\mu_t,\alpha}_s,\mu^\alpha_s) \right) \mathrm{d}s \right].
    \end{align*}
    Let $\varepsilon>0$, we have $\int_0^T c_t \left((t+\varepsilon) \wedge T \right) \mathrm{d}t \le 0.$
    By dividing the previous inequality by $\varepsilon$ and letting $\varepsilon \to 0$, we get that
    \begin{align*}
        &\int_0^T\E \left[  h \left( t, X^{t,\mu_t,\alpha}_t, \mu^\alpha_t, \widehat{Z}\left(t,X^{t,\mu_t,\alpha}_t,\mu^\alpha_t\right), \alpha(t,X^{t,\mu_t,\alpha}_t) \right) - h \left( t, X^{t,\mu_t,\alpha}_t, \mu^\alpha_t, \widehat{Z}\left(t,X^{t,\mu_t,\alpha}_t,\mu^\alpha_t\right), \widehat{\alpha}(t,X^{t,\mu_t,\alpha}_t,\mu^\alpha_t) \right) \right] \mathrm{d}t
        \\
        &=\int_0^T\int_\R  h \left( t, x, \mu_t, \widehat{Z}\left(t,x,\mu_t\right), \alpha(t,x) \right) - h \left( t, x, \mu_t, \widehat{Z}\left(t,x,\mu_t\right), \widehat{\alpha}(t,x,\mu_t) \right) \mu_t(\mathrm{d}x)\; \mathrm{d}t
        \le 0.
    \end{align*}
    By taking the supremum over $\alpha,$ we get that 
    \begin{align*}
        \int_0^T\int_\R  h \left( t, x, \mu_t, \widehat{Z}\left(t,x,\mu_t\right), \widehat{a}\left(t,x,\mu_t, \widehat{Z}\left(t,x,\mu_t\right)  \right) \right) - h \left( t, x, \mu_t, \widehat{Z}\left(t,x,\mu_t\right), \widehat{\alpha}(t,x,\mu_t) \right) \mu_t(\mathrm{d}x) \;\mathrm{d}t
        \le 0.
    \end{align*}

    By definition of $\widehat{a}$, we have $h(s,x,m,z,a) \le h(s,x,m,z,\widehat{a}(s,x,m,z))$. Therefore, by uniqueness of the maximizer of the Hamiltonian, we deduce that, a.e., $\widehat{\alpha}(t,x,\mu_t)=\widehat{a}\left(t,x,\mu_t, \widehat{Z} (t,x,\mu_t)\right)$ $\mu_t(\mathrm{d}x)\mathrm{d}t$--a.e.. We can conclude the result.
\end{proof}

\subsection{PDEs over the set of probability measures}

\subsubsection{Proof of \Cref{prop:PDE:fromBSDE}}

\begin{proof}
    We define the metric $\Wc_0(\nu , \nu') 
		~:=~
			\inf_{\lambda \in \Lambda(\nu, \nu')}  \int_{\R \x \R} |e-e'| \wedge1 ~\lambda( \mathrm{d}e, \mathrm{d}e')$ for any $(\nu,\nu') \in \Pc(\R) \x \Pc(\R)$ with $\Lambda(\mu, \mu')$ denotes the collection of all probability measures $\lambda$ on $\R \x \R$ 
	such that $\lambda( \mathrm{d}e, \R) = \mu$ and $\lambda(\R,  \mathrm{d}e') = \mu'( \mathrm{d}e')$. This metric induces the weak convergence topology. 

    \medskip
    We only show the sub--solution part, the super--solution property follows by identical argument. Let $\varphi: [0,T] \x \Pc(\R) \to \R$ be a smooth function s.t. the map $\widehat{u}-\varphi$ has a maximum at $(t_0,\nu_0)$. We assume that $\widehat{u}(t_0,\nu_0)=\varphi(t_0,\nu_0)$. We want to prove that
    \begin{align*}
        \partial_t \varphi(t_0,\nu_0) + H\left(t_0,\nu_0,\varphi,\partial_x \delta_m \varphi, \partial^2_x \delta_m \varphi \right) \ge 0.
    \end{align*}
    Let us assume that this is not true. By continuity, there exists $\delta >0$ and $\eta \in (0,T-t_0)$ s.t. $\widehat{u}(t,\nu) \le \varphi (t,\nu)$ and
     \begin{align} \label{eq:contradiction}
        \partial_t \varphi(t,\nu) + H\left(t,\nu,\varphi,\partial_x \delta_m \varphi, \partial^2_x \delta_m \varphi \right) \le -\delta
    \end{align}
    for all $t \in [t_0,t_0+\eta]$ and $\nu$ s.t. $\Wc_0(\nu_0,\nu) \le \eta$. Set $h(t,x,m):=-\partial_t \varphi(t,m)-\frac{1}{2} \partial_x^2 \delta_m \varphi(t,m)(x)$. We can see that $\left(\varphi, \partial_x \delta_m \varphi \right)$ is a W--BSDE solution associated to $\left( h,\varphi(t_0+\eta,\cdot)\right)$ when the initial time is $t_0$ and the maturity is $t_0+\eta$. 
    
\medskip    
    We recall that $(\mu_t)_{t \in [0,T]}$ belongs to  $\Lc^{t_0}$ if $\mu_s=\Lc(R_s)$ where $\mathrm{d}R_s=B(s,R_s)\mathrm{d}s + \mathrm{d}W_s$ for $s \in [t_0,T]$, and $\Lc(R_s)=\mu_{t_0}$ for $s \in [0,t_0]$. Also, $\left( Y^{t_0},Z^{t_0} \right)$ is a W--BSDE solution associated to $\left( f,\psi \right)$ with initial time $t_0$ and maturity $T$ and $\left( Y^{t_0 + \varepsilon},Z^{t_0 + \varepsilon} \right)$ is a W--BSDE solution associated to $\left( f,\psi \right)$ with initial time $t_0 + \varepsilon$ and maturity $T$. Therefore, on $[t_0+\varepsilon,T]$,  $\left( Y^{t_0},Z^{t_0} \right)$ and $\left( Y^{t_0 + \varepsilon},Z^{t_0 + \varepsilon} \right)$ satisfy the same type of {W--BSDE}, under the assumptions we assumed here, we can apply \Cref{cor:uniqueness} and find that $Y^{t_0}(t_0+\varepsilon,\mu_{t_0+\varepsilon})=Y^{t_0+\varepsilon}(t_0+\varepsilon,\mu_{t_0+\varepsilon})$ for any $\mu \in \Lc^{t_0}$ and any $\varepsilon>0$.

\medskip    
    Let $B:[0,T] \x \R \to \R$ be a bounded Borel map. We consider $(R_s)_{s \in [0,T]}$ satisfying: for $s \in [0,t_0]$ $\Lc(R_s)=\nu_0$ and for $s \in [t_0,T]$, $\mathrm{d}R_s= B(s,R_s)\mathrm{d}s + \mathrm{d}W_s$ with $\nu_s=\Lc(R_s)$. Let us observe that
    \begin{align} \label{eq:growth}
        \Wc_0(\nu_s,\nu_0) \le \E[|R_s-R_0|] \le (s-t_0)|B|_\infty+\sqrt{(s-t_0)}.
    \end{align}
    Consequently, for any $B$, there exists $s_0 \in [t_0,T]$ small enough s.t. for any $s \in [0,s_0]$, $\Wc_0(\nu_0,\nu_s) \le \eta$.
    

\medskip
    We have considered that $\widehat{u}(t,\nu) \le \varphi(t,\nu)$ i.e. $Y^{t}(t,\nu) \le \varphi(t,\nu)$ for $t \in [t_0,t_0+\eta]$ and $\Wc_0(\nu,\nu_0) \le  \eta$, and by \eqref{eq:contradiction},
    $$
        \int_\R f\left(t,x,\varphi(t,\nu), \partial_x \delta_m \varphi (t,\nu)(x) \right) \nu(\mathrm{d}x) \le \int_\R h(t,x,\nu) \nu(\mathrm{d}x) - \delta. 
    $$
     We know that $Y^{t_0}(t_0+\varepsilon,\mu_{t_0+\varepsilon})=Y^{t_0+\varepsilon}(t_0+\varepsilon,\nu_0) \le \varphi(t_0+\varepsilon,\nu_0)$ with $\varepsilon \le \eta$. By using the comparison principle of \Cref{prop:comparision} on $(Y^{t_0},Z^{t_0})$ and $(\varphi,\partial_x \delta_m \varphi)$ where we restrict ourselves over $\Lc^{t_0}$ and using the observation of \eqref{eq:growth} i.e. taking a small time to satisfy the condition $\Wc_0(\nu_0,\nu_s) \le \eta$, there exists $\varepsilon_0 >0$ small enough (depending on the Lipschitz property of $f$ see proof of \Cref{prop:comparision} ),  we get that 
    $$
        Y^{t_0}(t_0 +\varepsilon,\mu_{t_0+\varepsilon}) \le \varphi(t_0 +\varepsilon,\mu_{t_0+\varepsilon})-\delta
    $$
    for any $\varepsilon \in [0,\varepsilon_0 \wedge \eta]$ and $\mu \in \Lc^{t_0}$. In particular, $Y^{t_0}(t_0 ,\nu_{0}) \le \varphi(t_0,\nu_{0})-\delta$.  As $\delta>0$, this is a contradiction because we know that $\varphi(t_0,\nu_0)=Y^{t_0}(t_0,\nu_0)$.  We deduce that $Y$ is a sub--solution. The proof of super--solution is identical.
\end{proof}

\subsection{Comparison and uniqueness results}
\subsubsection{Proof of \Cref{prop:comparision}} \label{proof:prop_comparison}

\begin{proof}
    Let us observe that: for any $\mu \in \Lc^{t_1}$, for all $t \in [t_2,T]$,
    \begin{align*}
        Y^{t_2}(t,\mu_t) - Y^{t_2}(t,\mu_t) &= \psi^1(\mu_T) - \psi^2(\mu_T) 
        \\
        &~~~~+ \E \left[ \int_t^T f^1 \left( s, X_s,\mu_s, Y^{t_1}(s,\mu_s), Z^{t_1}(s,X_s,\mu_s) \right) - f^2 \left( s, X_s,\mu_s,Y^{t_2}(s,\mu_s), Z^{t_2}(s,X_s,\mu_s) \right) \mathrm{d}s \right] 
        \\
        &~~~~- \E \left[ \int_t^T \left(Z^{t_1}(s,X_s,\mu_s) - Z^{t_2}(s,X_s,\mu_s) \right) B(s,X_s) \mathrm{d}s \right].
    \end{align*}
    Let $\beta:[0,T] \x \R \to \R$ be a bounded Borel map. We set 
    \begin{align*}
        &F(t,x,m)
        \\
        &:= \frac{f^1\left( t,x,m,Y^{t_2}(t,m), Z^{t_1}(t,x,m) \right) - f^1\left( t,x,m,Y^{t_2}(t,m), Z^{t_2}(t,x,m) \right)}{Z^{t_1}(t,x,m)- Z^{t_2}(t,x,m)} \1_{Z^{t_2}(t,x,m) \neq Z^{t_1}(t,x,m)}+\beta(t,x).
    \end{align*}
    Notice that since $f^i$ is Lipschitz in $(y,z)$ uniformly in $(t,x,m)$, we can check that $F$ is bounded. Also, under the assumptions of the proposition, we have that $m \mapsto F(t,x,m)$ is continuous for each $(t,x)$. For each $t$, let us introduce $(S_s)_{s \in [0,T]}$ the process satisfying: $\Lc(S_s)=\mu_s$ for each $s \in [0,t_2]$, and for each $s \in [t_2,T]$,
    \begin{align*}
        \mathrm{d}S_s = F(s,S_s, \nu_s) \mathrm{d}s + \mathrm{d}W_s\;\mbox{with}\;\Lc(S_s)=\nu_s.
    \end{align*}
    Since $\mu \in \Lc^{t_1}$, similar to the proof of \Cref{prop:McK-fromBSDE}, we get $\left( \Lc(S_s) \right)_{s \in [0,T]} \in \Lc^{t_1}$. Therefore, for $t \in [t_2,T]$,
    \begin{align*}
        &Y^{t_1}(t,\nu_t) - Y^{t_2}(t,\nu_t) 
        \\&= \psi^1(\nu_T) - \psi^2(\nu_T)
        \\
        &~~~~+ \E \left[ \int_t^T f^1 \left( s, S_s,\nu_s,Y^{t_1}(s,\nu_s), Z^{t_1}(s,S_s,\nu_s) \right) - f^2 \left( s, S_s,\nu_s,Y^{t_2}(s,\nu_s), Z^{t_2}(s,S_s,\nu_s) \right) \mathrm{d}s \right] 
        \\
        &~~~~- \E \left[ \int_t^T \left(Z^{t_1}(s,S_s,\nu_s) - Z^{t_2}(s,S_s,\nu_s) \right) F(s,S_s,\nu_s) \mathrm{d}s \right]
        \\
        &=\psi^1(\nu_T) - \psi^2(\nu_T)
        \\
        &~~~~+ \E \left[ \int_t^T f^1 \left( s, S_s,\nu_s,Y^{t_1}(s,\nu_s), Z^{t_1}(s,S_s,\nu_s) \right) - f^2 \left( s, S_s,\nu_s,Y^{t_2}(s,\nu_s), Z^{t_2}(s,S_s,\nu_s) \right) \mathrm{d}s \right] 
        \\
        &~~~~- \E \left[ \int_t^T f^1\left(t,S_s,\nu_s,Y^{t_2}(s,\nu_s),Z^{t_1}(s,S_s,\nu_s) \right) - f^1\left(t,S_s,\nu_s,Y^{t_2}(s,\nu_s),Z^{t_2}(s,S_s,\nu_s) \right) \mathrm{d}s \right]
        \\
        &~~~~- \E \left[ \int_t^T \left(Z^{t_1}(s,S_s,\nu_s) - Z^{t_2}(s,S_s,\nu_s) \right) \beta(s,S_s) \mathrm{d}s \right]
        \\
        &=\psi^1(\nu_T) - \psi^2(\nu_T) 
        \\
        &~~~~+ \E \left[ \int_t^T f^1 \left( s, S_s,\nu_s,Y^{t_2}(s,\nu_s), Z^{t_2}(s,S_s,\nu_s) \right) - f^2 \left( s, S_s,\nu_s,Y^{t_2}(s,\nu_s), Z^{t_2}(s,S_s,\nu_s) \right) \mathrm{d}s \right]
        \\
        &~~~~+ \E \left[ \int_t^T \left(Y^{t_1}(s,\nu_s)-Y^{t_2}(s,\nu_s) \right) \frac{ f^1 \left( s, S_s,\nu_s,Y^{t_1}(s,\nu_s), Z^{t_1}(s,S_s,\nu_s) \right) - f^1 \left( s, S_s,\nu_s,Y^{t_2}(s,\nu_s), Z^{t_1}(s,S_s,\nu_s) \right)}{Y^{t_1}(s,\nu_s)-Y^{t_2}(s,\nu_s)} \mathrm{d}s \right]
        \\
        &~~~~- \E \left[ \int_t^T \left(Z^{t_1}(s,S_s,\nu_s) - Z^{t_2}(s,S_s,\nu_s) \right) \beta(s,S_s) \mathrm{d}s \right].
    \end{align*}
    We set $\ell_r:=\E \left[\int_t^r  \frac{ f^1 \left( s, S_s,\nu_s,Y^{t_1}(s,\nu_s), Z^{t_1}(s,S_s,\nu_s) \right) - f^1 \left( s, S_s,\nu_s,Y^{t_2}(s,\nu_s), Z^{t_1}(s,S_s,\nu_s) \right)}{Y^{t_1}(s,\nu_s)-Y^{t_2}(s,\nu_s)} \mathrm{d}s \right]$ for each $r \in [t,T]$. Since $f^1$ is Lipschitz in $(y,z)$ uniformly in $(t,m,x)$, $\ell_t$ is bounded uniformly in $t$. We define $D_s:=\left(Y^{t_1}(s,\nu_s)-Y^{t_2}(s,\nu_s) \right)e^{\ell_s}$. Then
    \begin{align*}
        \mathrm{d}D_s
        =
        &-e^{\ell_s} \E\left[ f^1 \left( s, S_s,\nu_s,Y^{t_2}(s,\nu_s), Z^{t_2}(s,S_s,\nu_s) \right) - f^2 \left( s, S_s,\nu_s,Y^{t_2}(s,\nu_s), Z^{t_2}(s,S_s,\nu_s) \right) \right] \mathrm{d}s
        \\
        &+e^{\ell_s}\E \left[\left(Z^{t_1}(s,S_s,\nu_s) - Z^{t_2}(s,S_s,\nu_s) \right) \beta(s,S_s)  \right]\mathrm{d}s.
    \end{align*}
    Consequently,
    \begin{align} \label{eq:comparison_Z}
        D_t=&e^{\ell_T} \left( \psi^1(\nu_T) - \psi^2(\nu_T) \right)-\E \left[\int_t^Te^{\ell_s}\left(Z^{t_1}(s,S_s,\nu_s) - Z^{t_2}(s,S_s,\nu_s) \right) \beta(s,S_s) \mathrm{d}s \right] \nonumber 
        \\
        &+  \int_t^T e^{\ell_s} \E \left[ f^1 \left( s, S_s,\nu_s,Y^{t_2}(s,\nu_s), Z^{t_2}(s,S_s,\nu_s) \right) - f^2 \left( s, S_s,\nu_s,Y^{t_2}(s,\nu_s), Z^{t_2}(s,S_s,\nu_s) \right) \right] \mathrm{d}s. 
    \end{align}
    Since $e^{\ell_s} \ge 0$, by using the inequalities 
    $$
        \int_\R f^1(t,x,\eta_t,Y^{t_2}(t,\eta_t),Z^{t_2}(t,x,\eta_t)) \eta_t(\mathrm{d}x) \le \int_\R f^{2}(t,x,\eta_t,Y^{t_2}(t,\eta_t),Z^{t_2}(t,x,\eta_t)) \eta_t(\mathrm{d}x)\;\mathrm{d}t\mbox{--a.e.}
    $$
     and $\psi^1(\eta_T) \le \psi^2(\eta_T)$ for any $\eta \in \Lc$, if $\beta=0$, we deduce that $Y^{t_1}(t,\mu_t) \le Y^{t_2}(t,\mu_t)$. This is true for any $t \in [t_2,T]$ and $\mu$. 

    \medskip
    
\end{proof}

\subsubsection{Proof of \Cref{cor:uniqueness}} \label{proof:corr_uniqueness}

\begin{proof}
    The equality regarding $Y^{t_1}$ and $Y^{t_2}$ is straightforward.  Let us now turn our attention to the equality involving $Z^{t_1}$ and $Z^{t_2}$. We keep the notations of the proof of \Cref{prop:comparision} above in \Cref{proof:prop_comparison}.  Using the condition $f^1(t,x,\eta_t,Y^{t_2}(t,\eta_t),Z^{2}(t,x,\eta_t)) = f^{t_2}(t,x,\eta_t,Y^{t_2}(t,\eta_t),Z^{t_2}(t,x,\eta_t))$ $\eta_t(\mathrm{d}x)\mathrm{d}t$--a.e. for any $\eta \in \Lc^{t_2}$, we have $F(t,x,\nu_t)=\beta(t,x)$ a.e. $(t,x)$ and $\ell_s=0$. Using the equation proved in \eqref{eq:comparison_Z} with any bounded Borel $\beta$ combined with and $\psi^1(\eta_T) = \psi^2(\eta_T)$ for any $\eta \in \Lc^{t_2}$,  we obtain, for any $t \in [t_2,T]$,
    \begin{align*}
        \E \left[\int_t^T\left(Z^{t_1}(s,S_s,\nu_s) - Z^{t_2}(s,S_s,\nu_s) \right) \beta(s,S_s) \mathrm{d}s \right]=0.
    \end{align*}+
    This leads to $\left(Z^{t_1}(s,S_s,\nu_s) - Z^{t_2}(s,S_s,\nu_s) \right) \beta(s,S_s)=0$ $\mathrm{d}s \otimes\mathrm{d}\P$--a.e. $(s,\om) \in [t_2,T] \x \Om$. Since $S$ is a non--degenerate diffusion, we deduce that $\left(Z^{t_1}(s,x,\nu_s) - Z^{t_2}(s,x,\nu_s) \right) \beta(s,x)=0$ $\mathrm{d}s \otimes\mathrm{d}x$--a.e. $(s,x) \in [t_2,T] \x \R$. Since $F(t,x,\nu_t)=\beta(t,x)$, $(\nu_t)_{t \in [0,T]}$ belongs to $\Lc^{t_1}$ and its associated to the bounded Borel drift $\beta$. This is enough to conclude the result.
\end{proof}

\subsection{Existence results}

\subsubsection{Proof of \Cref{prop:existY}}

\begin{proof}
    \textbf{Step 1 : Existence for small maturity}
    By \Cref{prop:existence:Y}, there exists small $T>0$ and a unique Borel map $Y:[0,T] \x \Pc(\R) \to \R$ s.t. for a random variable $U$ independent of $W$ with $\Lc(U)=m$,
    \begin{align*}
        Y(t,m)
        =
        \psi\left( \Lc(X^{t,U}_T) \right) + \int_t^T h \left( r, \Lc(X^{t,U}_r), Y\left(r, \Lc(X^{t,U}_r) \right)\; \right)\;\mathrm{d}r,
    \end{align*}
    where for $s \in [t,T]$, $\mathrm{d}X^{t,x}_s= {P}(s, X^{t,x}_s)\mathrm{d}s +\mathrm{d}W_s$ and for $s \in [0,t]$, $X^{t,x}_s=x$. In addition, $m \mapsto Y(t,m)$ admits a linear functional derivative $\delta_m Y(t,m)(x)$ s.t. $\sup_{(t,x,m)}|\delta_m Y(t,m)(x)| < \infty$. Now, by an application of \Cref{corollary:FK_measure_simplified}, we find, with $\mu=(\Lc(U_t))_{t \in [0,T]} \in \Lc$, for $r \ge t$,
    \begin{align*}
         &h \left(r, \Lc\left(X^{t,U_t}_r \right), Y\left( r, \Lc(X^{t,U_t}_r) \right) \right)
         \\&=h\left(r,\Lc(U_r), Y(r,\Lc(U_r)) \right) 
         + \E \left[ \int_t^r \delta_m \widehat{h}\left(r, \Lc\left( X^{s,U_s}_r \right)  \right) \left( X^{s,U_s}_r \right) \frac{\partial_xp(r,X^{s,U_s}_r,s,U_s)}{p(r,X^{s,U_s}_r,s,U_s)} \left( P(s,U_s) - B(s,U_s) \right) \mathrm{d}s \right]
    \end{align*}
    where $\widehat{h}(r,m):= h \left( r, m, Y\left(r, m \right) \right)$, and 
    $$
        \psi\left( \Lc(X^{t,U}_T) \right)=\psi(\Lc(U_T)) + \E \left[ \int_t^T \delta_m \psi\left(\Lc\left( X^{s,U_s}_T \right)  \right) \left( X^{s,U_s}_T \right) \frac{\partial_xp(T,X^{s,U_s}_T,s,U_s)}{p(T,X^{s,U_s}_T,s,U_s)} \left( P(s,U_s) - B(s,U_s) \right) \mathrm{d}s \right].
    $$
    With $\Lc(U)=m$, we set 
    \begin{align*}
        Z(t,x,m)
        :=
        ~&\E \left[ \int_t^T \delta_m \widehat{h}\left(r, \Lc\left( X^{t,U}_r \right)  \right) \left( X^{t,U}_r \right) \frac{\partial_xp(r,X^{t,U}_r,t,U_t)}{p(r,X^{t,U}_r,t,U)} \mathrm{d}r\bigg| U=x \right]
        \\
        &~~~~~~+ \E \left[ \delta_m \psi\left(\Lc\left( X^{t,U}_T \right)  \right) \left( X^{t,U}_T \right) \frac{\partial_xp(T,X^{t,U}_T,t,U)}{p(T,X^{t,U}_T,t,U)}\bigg| U=x\right]. 
    \end{align*}
    We can then check that $(Y,Z)$ is a {\rm W--BSDE} solution associated to $(f,\psi)$ (the integrability condition of $Z$ can be verified with \Cref{remark:integrability_Z}).

    \medskip
    \textbf{Step 2 : Existence for any maturity}. By the previous step, there exists small $T_\circ >0$ s.t. we can find $(Y,Z)$ a {\rm W--BSDE} solution associated to $(f,\psi)$ for maturity $T_\circ$. Let $T>0$.  we divide the interval $[0,T]$ into $n:=\left\lceil \frac{T}{T_\circ} \right\rceil$ small intervals $([t_i,t_{i+1}])_{i=0,\cdots,n-1}$. By the previous step, there exists $(Y^n,Z^n)$ a {\rm W--BSDE} solution associated to $(f,\psi)$ over $[t_{n-1},t_n]$. Since $m \mapsto Y^n(t,m)$ admits a linear functional derivative $\delta_m Y^n$ verifying $\sup_{(t,m,x)} |\delta_m Y^n(t,m)(x)| < \infty$, we can again apply the previous step and find $(Y^{n-1},Z^{n-1})$ associated to $(f,Y^n(t_{n-1},\cdot))$ over $[t_{n-2}, t_{n-1}]$. We repeat this procedure and construct a sequence $(Y^i,Z^i)_{1 \le i \le n}$ s.t. $(Y^i,Z^i)$ is a {\rm W--BSDE} solution associated to $(f,Y^{i+1}(t_i,\cdot))$ over $[t_{i-1},t_i]$. We set $Y(t,m):=\sum_{i=1}^n Y^i(t,m) \mathbf{1}_{t \in [t_{i-1},t_i]}$ and $Z(t,x,m):=\sum_{i=1}^n Z^i(t,x,m) \mathbf{1}_{t \in [t_{i-1},t_i]}$. We can check that $(Y,Z)$ is a {\rm W--BSDE} solution associated to $(f,\psi)$ over $[0,T]$.
\end{proof}

\subsubsection{Proof of \Cref{prop:existSquare}} \label{sec:proof:square}

\begin{proof}
    \textbf{Step 1 : the associated $n$--particle representation}
    Let $\mu \in \Lc$ associated to $B$. Let $(W^i)_{i \ge 1}$ be a sequence of independent Brownian motions and $(X^i)_{i \ge 1}$ be a sequence of i.i.d processes with $\Lc(X^i)=\Lc(X)$ where $X$ is a process s.t. $\mu_t=\Lc(X_t)$ for each $t \in [0,T]$. For each $n \ge1$, we consider the $\F$--adapted processes $(\overline{Y}^n,Z^1,\dots,Z^n)$ that satisfy (see for instance \cite{kobylanski2000backward}): a.e., for each $t \in [0,T]$,
    \begin{align*}
        \overline{Y}^n_t = n\psi(\mu^n_T) + \int_t^T \frac{1}{2} \sum_{i=1}^n \left| n Z^i_s \right|^2 \mathrm{d}s -\int_t^T \sum_{i=1}^n n Z^i_s \mathrm{d}W^i_s\;\;\mbox{with}\;\;\mu^n_t:=\frac{1}{n} \sum_{i=1}^n \delta_{W^i_t+X^i_0}.
    \end{align*}
    We set $P^n_t:=e^{\overline{Y}^n_t}$ and $Q^i_t:=n Z^i_t P^n_t $. By It\^o formula, the tuple $\left( P^n, Q^1,\dots,Q^n \right)$ satisfies $\mathrm{d}P^n_t=\sum_{i=1}^n Q^i_t \mathrm{d}W^i_t$  with $P^n_T=e^{n \psi (\mu^n_T)}.$ Therefore, we check that
    \begin{align*}
        P^n_t
        =
        \E \left[ e^{n \psi\left(\mu^n_T \right)} \Big| \Fc_t \right]\;\mbox{and}\;Q^i_t=\E \left[\partial_x \delta_m \psi \left(\mu^n_T \right)(X^i_0+W^i_T) e^{n \psi\left(\mu^n_T \right)} \Big| \Fc_t \right].
    \end{align*}
    This leads to 
    \begin{align*}
        nZ^i_t
        =
        \frac{\E \left[\partial_x \delta_m \psi \left(\mu^n_T \right)(X^i_0+W^i_T) e^{n \psi\left(\mu^n_T \right)} \Big| \Fc_t \right]}{\E \left[ e^{n \psi\left(\mu^n_T \right)} \Big| \Fc_t \right]}.
    \end{align*}

    \textbf{Step 2 : convergence of Gibbs measures} By the law of large number, there exists $A \in \Fc$ s.t. $\P(A)=1$ and for any $\om \in A$, $\lim_{n \to \infty}\frac{1}{n}\sum_{i=1}^n \delta_{W^i(\om)}=\Lc(W^1)$ and $\lim_{n \to \infty}\frac{1}{n}\sum_{i=1}^n \delta_{X^i(\om)}=\Lc(X^1)$. Let $\om \in A$, we introduce 
    \begin{align} \label{eq:gibbs_measure_n}
        \mu^{n,\om}_t(\mathrm{d}z^1,\cdots,\mathrm{d}z^n):=\frac{1}{Z^{n,\om}_t} e^{n \psi \left( \frac{1}{n} \sum_{j=1}^n \delta_{z^j + X^j_t(\om)} \right)} \Lc(W_{T-t})(\mathrm{d}z^1) \cdots \Lc(W_{T-t})(\mathrm{d}z^n)
    \end{align}
    $\;\mbox{with}\;Z^{n,\om}_t:=\int_{\R^n}  e^{n \psi \left( \frac{1}{n} \sum_{j=1}^n \delta_{z^j + X^j_t(\om)} \right)} \Lc(W_{T-t})(\mathrm{d}z^1) \cdots \Lc(W_{T-t})(\mathrm{d}z^n)$.

    \medskip
    Let $\phi:\R^2 \to \R$ be a Borel bounded map. By the law of large numbers, 
    \begin{align*}
        &\frac{1}{n} \ln \left( \E \left[ e^{\sum_{j=1}^n \phi\left( W^j_T-W^j_t, X^j_t(\om) \right)} \right] \right)=  \frac{1}{n} \ln \left( \prod_{j=1}^n \E \left[ e^{ \phi\left( W^j_T-W^j_t, X^j_t(\om) \right)} \right] \right) = \frac{1}{n} \ln \left( \prod_{j=1}^n \E \left[ e^{ \phi\left( W^1_T-W^1_t, X^j_t(\om) \right)} \right] \right)
        \\
        &= \frac{1}{n} \sum_{j=1}^n \ln \left( \E \left[ e^{ \phi\left( W^1_T-W^1_t, X^j_t(\om) \right)} \right] \right) \to_{n \to \infty} \int_\R \ln \left( \E \left[ e^{\phi\left( W^1_T-W^1_t, x \right)} \right] \right) \mu_t(\mathrm{d}x)=:\Lambda(\phi).
    \end{align*}
    We follow the reasoning of \cite[Proof of Theorem 6.2.10]{Dembo2010} which is based on the application of \cite[Corollary 4.6.11]{Dembo2010}. For any $q \ge 1$ and any bounded Borel maps $f^1,\cdots,f^q$, the map $(t_1,\cdots,t_q) \to \Lambda\left( \sum_{i=1}^q t_i f^i \right)$ essentially smooth, lower semicontinuous, and finite in some neighborhood of $0.$ We can apply \cite[Corollary 4.6.11]{Dembo2010} and deduce that $\left( \Lc\left( \frac{1}{n} \sum_{i=1}^n \delta_{\left( W^{i}_T-W^i_t, X^i_t(\om) \right)} \right) \right)_{n \ge 1}$ satisfies a large deviation principle with the convex, good rate function
    \begin{align} \label{eq:optim}
        \Lambda^*(m):=\sup_{\phi} \left[ \int_{\R^2} \phi(x,y) m(\mathrm{d}x,\mathrm{d}y) -\Lambda(\phi) \right]
    \end{align}
    where the supremum is over the set of bounded Borel maps. Let us provide a better form for $\Lambda^\star$. In fact, we can check that if $m \in \Pc(\R^2)$ is such that $m(\R, \mathrm{d}x) \neq \mu_t(\mathrm{d}x)$, by choosing a map independent of the first variable in the optimization \eqref{eq:optim}, we have $\Lambda^\star(m)=\infty$. When $m(\R, \mathrm{d}x) = \mu_t(\mathrm{d}x)$, with $x \mapsto m(x,\mathrm{d}y)$ the Borel map s.t. $m(\mathrm{d}y,\mathrm{d}x)=m(x,\mathrm{d}y)\mu_t(\mathrm{d}x)$, we get
    \begin{align*}
        &\Lambda^\star(m)= \sup_{\phi} \left[ \int_{\R^2} \phi(x,y) m(\mathrm{d}x,\mathrm{d}y) -\Lambda(\phi) \right] = \sup_{\phi} \left[ \int_{\R^2} \phi(x,y) m(y,\mathrm{d}x)\mu_t\mathrm{d}y) -\Lambda(\phi) \right]
        \\
        &=
        \int_\R\sup_{\phi} \left[ \int_{\R} \phi(x,y) m(y,\mathrm{d}x) - \ln \left( \E \left[ e^{\phi\left( W^1_T-W^1_t, y \right)} \right] \right)  \right] \mu_t(\mathrm{d}y).
    \end{align*}
    By playing over the set of optimization again, we can see that if $m(y,\mathrm{d}x)$ is not independent of $y$, we have $\Lambda^\star(m)=\infty$. Using \cite[Lemma 6.2.13]{Dembo2010} (Donsker--Varadhan variational formula), we finally obtain
    \begin{align*}
        \Lambda^\star(m)=D(m^1|\Lc(W_T-W_t))\;\mbox{if}\;m(\mathrm{d}x,\mathrm{d}y)=m^1(\mathrm{d}x)\mu_t(\mathrm{d}y)\;\mbox{and}\;\Lambda^\star(m)=\infty\;\mbox{otherwise}
    \end{align*}
    where we recall that
    \begin{align*}
        D(m|\mu):=\int_{\R \x \R} \frac{\mathrm{d}m}{\mathrm{d}\mu} \log \left( \frac{\mathrm{d}m}{\mathrm{d}\mu}  \right) \mathrm{d}\mu\mbox{ if }m \ll\mu,\;D(m|\mu)=\infty\mbox{ otherwise}.
    \end{align*}
    We define the map $\widehat{\psi}:\Pc(\R^2) \to \R$ by $\widehat{\psi}(m):=\psi (\Lc(X+Y))$ where $\Lc(X,Y)=m$. Since $\psi$ is continuous bounded, by Varadhan’s Integral Lemma (see \cite[Theorem 4.3.1]{Dembo2010}), we find that
    \begin{align} \label{eq:cong_Y}
         \lim_{n \to \infty} \frac{1}{n} \ln \left( \E \left[ e^{n \psi\left(\frac{1}{n} \sum_{j=1}^n \delta_{W^j_{T}-W^j_t+X^j_t(\om)} \right)}  \right]  \right)= \sup_{m \in \Pc(\R^2)}\widehat{\psi} (m)-\Lambda^\star(m)=\sup_{m \in \Pc(\R)} \psi(m*\mu_t) -D\left(m\;|\;\Lc(W^1_{T-t}) \right).
    \end{align}
    Setting $\R^n \ni (x^1,\cdots,x^n) \mapsto L^n(x^1,\cdots,x^n):=\frac{1}{n} \sum_{i=1}^n \delta_{x^i}$,  by the tilting property,  the sequence $\left( \mu^{n,\om}_t \circ (L^n)^{-1} \right)_{n \ge 1}$ satisfies a large deviation principle with the rate function  
    \begin{align*}
        F_t(m):=\Lambda^\star(m)-\widehat{\psi} (m)+\sup_{m' \in \Pc(\R^2)}\widehat{\psi} (m')-\Lambda^\star(m')=\Lambda^\star(m)-\widehat{\psi} (m)+\sup_{m \in \Pc(\R)} \psi(m*\mu_t) -D\left(m\;|\;\Lc(W^1_{T-t}) \right).
    \end{align*}
    If the map $m \mapsto \psi(m*\mu_t) -D\left(m\;|\;\Lc(W^1_{T-t})\right)$ admits a unique maximum $\nu_t$, then $\lim_{n \to \infty} \mu^{n,\om}_t \circ (L^n)^{-1}=\delta_{\nu_t}$.  Using \cite[Proposition 2.5.]{hu2020meanfieldlangevindynamicsenergy}, we know that $\nu_t$ exists and verifies
    \begin{align*}
        \nu_t(\mathrm{d}z)
        =
        \frac{\exp\left\{ \E \left[  \delta_m {\psi}(\nu_t*\mu_t)\left( z+X_t\right) \right] \right\} }{\int_{\R} \exp\left\{ \E \left[  \delta_m {\psi}(\nu_t*\mu_t)\left( z'+ X_t\right) \right]\right\} \Lc\left(W^1_{T-t}\right)(\mathrm{d}z') }\Lc\left(W^1_{T-t}\right)(\mathrm{d}z).
    \end{align*}

    \textbf{Step 3 : convergence of BSDEs through Gibbs measures} We recall that $\mu^{n,\om}_t$ is defined in \eqref{eq:gibbs_measure_n}. Set $\mub^{n,\om}_t:=\mu^{n,\om}_t 1_{\om \in {A}} + \beta_0 1_{\om \in {A}^c}$ for some $\beta_0 \in \Pc(\R)$. It is straightforward to see that $(t,\om) \mapsto \mub^{n,\om}_t$ is Borel measurable. We set $\mathrm{d}\P^n=L^n_T\mathrm{d}\P$ where $\mathrm{d}L^n_t=L^n_t \sum_{i=1}^n B(t,X^i_0+W^i_t) \mathrm{d}W^i_t$. Since $B$ is bounded, by Girsanov Theorem, we have that $\widetilde{W^i_\cdot}:=W^i_\cdot-\int_0^\cdot B(s,X^i_0+W^i_s) \mathrm{d}s$ is a $\P^n$--Brownian motion. It is straightforward to check $\Lc\left( X^1,\dots,X^n \right)=\Lc^{\P^n}\left(W^1,\dots,W^n \right)$. Let $\varphi:[0,T] \x \R \to \R$ be a bounded continuous map. By straightforward computations, we get
    \begin{align*}
        &\frac{1}{n} \sum_{i=1}^n \E^{\P^n} \left[ \varphi(t,X^i_t) n Z^i_t \right]
        \\
        &= \int_{\Om}\frac{1}{n} \sum_{i=1}^n \int_{\R^n} \partial_x \delta_m \psi\left( \frac{1}{n} \sum_{i=1}^n \delta_{z^j + X^j_t(\om)} \right) \left(z^i + X^i_t(\om) \right) \varphi(t,X^i_t(\om)) \mub_t^{n,\om} (\mathrm{d}z^1,\cdots,\mathrm{d}z^n) \P(\mathrm{d}\om).
    \end{align*}
    Therefore, since $(X^i)_{i \ge 1}$ is i.i.d. and $\left(\mu^{n,\om}_t \right)_{n \ge 1}$ satisfies a large deviation principle i.e. $\lim_{n \to \infty} \mu^{n,\om}_t \circ (L^n)=\delta_{\nu_t}$, 
    \begin{align*}
        \lim_{n \to \infty} \frac{1}{n} \sum_{i=1}^n \E^{\P^n} \left[ \varphi(t,X^i_t) n Z^i_t \right]
        = \int_\R \int_{\R} \partial_x \delta_m {\psi}\left( \nu_t*\mu_t \right) \left(z+x \right) \varphi(t,x) \nu_t(\mathrm{d}z) \mu_t(\mathrm{d}x).
    \end{align*}
    In addition, we also have
    \begin{align*}
        &\frac{1}{n} \sum_{i=1}^n \E^{\P^n} \left[ \left| n Z^i_t \right|^2 \right]
        \\
        &=\int_{\R^n} \int_{\R^n} \frac{1}{n} \sum_{i=1}^n \partial_x \delta_m \psi\left( \frac{1}{n} \sum_{i=1}^n \delta_{z^j + X^j_t(\om)} \right) \left(z^i + X^i_t(\om) \right)
        \\
        &~~~~~~~~~~~~~~~~~~~~~~~~\partial_x \delta_m \psi\left( \frac{1}{n} \sum_{i=1}^n \delta_{v^j + X^j_t(\om)} \right) \left(v^i + X^i_t(\om) \right) \mub_t^{n,\om} (\mathrm{d}z^1,\cdots,\mathrm{d}z^n) \mub_t^{n,\om} (\mathrm{d}v^1,\cdots,\mathrm{d}v^n)\P(\mathrm{d}\om),
    \end{align*}
    then
    \begin{align*}
        \lim_{n \to \infty} \frac{1}{n} \sum_{i=1}^n \E^{\P^n} \left[ \left| n Z^i_t \right|^2 \right]
        &=\int_{\R} \int_{\R^2}  \partial_x \delta_m {\psi}\left( \nu_t*\mu_t \right) \left(z + x \right)
        \partial_x \delta_m \psi\left( \nu_t*\mu_t \right) \left(v + x \right) \nu_t(\mathrm{d}z) \nu_t(\mathrm{d}v) \mu_t(\mathrm{d}x)
        \\
        &=\int_{\R} \left| \int_{\R}  \partial_x \delta_m {\psi}\left( \nu_t*\mu_t \right) \left(z + x \right)
        \nu_t(\mathrm{d}z) \right|^2 \mu_t(\mathrm{d}x).
    \end{align*}
    Setting $Y^n:=\frac{\overline{Y}^n}{n}$, we have that $({Y}^n,Z^1,\dots,Z^n)$  satisfies: a.e., for each $t \in [0,T]$,
    \begin{align*}
        {Y}^n_t = \psi(\mu^n_T) + \int_t^T \frac{1}{2} \frac{1}{n} \sum_{i=1}^n \left| n Z^i_s \right|^2 \mathrm{d}s -\int_t^T \frac{1}{n}\sum_{i=1}^n n Z^i_s \mathrm{d}W^i_s\;\;\mbox{with}\;\;\mu^n_t:=\frac{1}{n} \sum_{i=1}^n \delta_{W^i_t}.
    \end{align*}
    Consequently,
    \begin{align*}
        \lim_{n \to \infty} \E^{\P^n} [Y^n_t]
        &= \psi(\mu_T) + \frac{1}{2}\int_t^T \int_{\R} \left| \int_{\R}  \partial_x \delta_m {\psi}\left( \nu_s*\mu_s \right) \left(z + x \right)
        \nu_s(\mathrm{d}z) \right|^2 \mu_s(\mathrm{d}x) \mathrm{d}x \mathrm{d}s 
        \\
        &~~~~~~~~~~~~~~~~- \int_t^T \int_\R \int_{\R} \partial_x \delta_m {\psi}\left( \nu_s*\mu_s \right) \left(z+x \right) B(s,x) \nu_s(\mathrm{d}z) \mu_s(\mathrm{d}x) \mathrm{d}s.
    \end{align*}
    We also proved that (see \eqref{eq:cong_Y})
    \begin{align*}
         \lim_{n \to \infty}\E^{\P^n}[Y^n_t] =  \lim_{n \to \infty} \E \left[\frac{1}{n} \ln \left( \E \left[ e^{n \psi\left(\frac{1}{n} \sum_{j=1}^n \delta_{W^j_{T}-W^j_t+X^j_t} \right)} \bigg| \Fc_t \right]  \right) \right] =\sup_{m \in \Pc(\R)} \psi(m*\mu_t) -D\left(m\;|\;\Lc(W^1_{T-t}) \right).
    \end{align*}
    Let $\nu_t=\nu(t,\mu_t)$, by defining 
    $$
        Y(t,m):=\sup_{m' \in \Pc(\R)} \psi(m'*m) -D\left(m'\;|\;\Lc(W^1_{T-t}) \right)\;\mbox{and}\;Z(t,x,m):= \int_{\R}  \partial_x \delta_m {\psi}\left( \nu(t,m)*m \right) \left(z + x \right)
        \nu(t,m)(\mathrm{d}z),
    $$
    we find, for all $t \in [0,T]$,
        \begin{align*}
            Y(t,\mu_t)
            =
            \psi(\mu_T) + \frac{1}{2}\E \left[\int_t^T |Z(s,X_s,\mu_s)|^2 \mathrm{d}s \right] - \E \left[\int_t^T Z(s,X_s,\mu_s) B(s,X_s) \mathrm{d}s \right].
        \end{align*}
\end{proof}

\begin{appendix}
\section{Technical results}

\subsection{A version of the It\^o's formula }

In this section, we derive the dynamics of a functional that depends in a specific manner on a flow of probability measures associated with a diffusion process.

\medskip
Let \( F : \mathcal{P}(\mathbb{R}) \to \mathbb{R} \) be a bounded continuous functional that admits a bounded and continuous linear functional derivative \( \delta_m F \), such that the map \( (x, m) \mapsto \delta_m F(m)(x) \) is differentiable in \( x \), and the function \( (x, m) \mapsto \partial_x \delta_m F(m)(x) \) is continuous.
We consider a measure flow \( \bar{\mu} \in \mathcal{L} \) associated with a bounded Borel measurable function \(  [0,T] \times \mathbb{R} \times \mathcal{P}(\mathbb{R}) \ni (t,x,m) \mapsto \overline{B}(t,x,m) \in \mathbb{R} \) Lipschitz in $(x,m)$ uniformly in $t$. The function \( \overline{B} \) is also assumed to be differentiable in the spatial variable \( x \), with bounded derivative denoted by \( \overline{B}' \). Moreover, for each fixed \( t \in [0,T] \), the map \( (x,m) \mapsto \overline{B}'(t,x,m) \) is continuous. In addition, \( \overline{B} \) admits a linear functional derivative with respect to the measure variable, denoted by \( \delta_m \overline{B}(t,x,m)(y) \), which is bounded and differentiable in \( y \). The derivative \( \partial_y \delta_m \overline{B}(t,x,m)(y) \) is further assumed to be bounded and continuous in the variables \( (x,m,y) \) for each fixed \( t \in [0,T] \).


\medskip
Let $\mu=\Lc(U) \in \Lc$, for each $(t,x)$, we denote by $\left(X^{t,x}_s \right)_{s \in [0,T]}$ the process satisfying: for each $s \in [0,T]$, a.e.
\begin{align*}
    X^{t,x}_s= x+ \int_t^{s \vee t} \overline{B}\left(r, X^{t,x}_r,\nu_r \right) \mathrm{d}r + W_{s \vee t}-W_t\;\mbox{with}\;\nu_s=\Lc\left(X^{t,U_t}_s \right).
\end{align*}
The process $\left(X^{t,x}_s \right)_{s \in [0,T]}$ is well--defined by \cite[Theorem A.3]{djete2019mckean}.
We denote by $(\overline{X}^{t,x}_s )_{s \in [0,T]}$ and $( \overline{U}_s )_{s \in [0,T]}$ the processes that satisfy the same equations as $(X^{t,x}_s)_{s \in [0,T]}$ and $(U_s)_{s \in [0,T]}$, respectively, but are driven by a Brownian motion $\overline{W}$ independent of $W$. In addition, we introduce the processes $(J^{t,U_t}_s)_{s \in [0,T]}$ and $(\overline{J}^{t,U_t,\overline{U}_t}_s)_{s \in [0,T]}$ by: for each $s \in [0,T]$, $J^{t,U_t}_s=1$, $\overline{J}^{t,U_t,\overline{U}_t}_s=0$ and, for each $s \in [t,T]$, $\muh_s=\Lc\left(\overline{J}^{t,U_t,\overline{U}_t}_s,\overline{X}^{t,\overline{U}_t}_s\mid X^{t,U_t}_{s \wedge \cdot} \right)$, 
$$
    \mathrm{d}J^{t,U_t}_s=J^{t,U_t}_s \overline{B}' (s,X^{t,U_t}_s,\nu_s) \mathrm{d}s
$$
and
    \begin{align*}
        \;\mathrm{d}\overline{J}^{t,U_t,\overline{U}_t}_s=\overline{J}^{t,U_t,\overline{U}_t}_s \overline{B}'\left(s,\overline{X}^{t,\overline{U}_t}_s,\nu_s \right) \mathrm{d}s &+ \int_{\R^2} j \partial_y \delta_m \overline{B} (s, \overline{X}^{t,\overline{U}_t}_s,\nu_s)(x) \muh_s(\mathrm{d}j,\mathrm{d}x) \mathrm{d}s
        \\
        &+ J^{t,U_t}_s \partial_y \delta_m \overline{B}(s, \overline{X}^{t,\overline{U}_t}_s,\nu_s)(X^{t,U_t}_s)\mathrm{d}s.
    \end{align*}

\begin{proposition} \label{prop:FK_measure_general}
    For any $\mu \in \Lc$ associated to $B$, and $(U_s)_{s \in [0,T]}$ the process s.t. $\Lc(U_s)=\mu_s$ for any $ s \in [0,T]$. We have for each $t \in [0,T]$,
    \begin{align*}
         F \left( \Lc\left( X^{t,U_t}_T \right) \right)=F\left(\Lc(U_T) \right) &+ \E \left[ \int_t^T \delta_m F\left( \Lc\left( X^{r,U_r}_T \right)  \right) \left( X^{r,U_r}_T \right) \frac{\partial_xp(T,X^{r,U_r}_T,r,U_r)}{p(T,X^{r,U_r}_T,r,U_r)} \left( \overline{B} (r,U_r,\mu_r) - B(r,U_r) \right) \mathrm{d}r \right]
         \\
         &+ \E \left[ \int_t^T \overline{J}^{r,U_r,\overline{U}_r}_T \partial_x \delta_m F\left( \Lc\left( X^{r,U_r}_T \right)  \right) \left( \overline{X}^{r,\overline{U}_r}_T \right) \left( \overline{B} (r,U_r,\mu_r) - B(r,U_r) \right) \mathrm{d}r \right] .
    \end{align*}
    where $\Lc(X^{t,x}_s)(\mathrm{d}y)=p(s,y,t,x)\mathrm{d}y$ with a Borel map $p:[0,T] \x \R \x [0,T] \x \R \to \R_+$ satisfying $\int_\R p(s,y,t,x) \mathrm{d}y=1$ for each $(s,t,x)$.
\end{proposition}

\begin{proof}
    The proof is divided into $3$ steps. We first assume that $\mu_0 \in \Pc_p(\R)$ for all $p>1$.

    \medskip
    \textbf{Step 1 : $n$--particle approximation} Let $\xbb=(x^1,\cdots,x^n) \in \R^n$, $\Xbb^{t,\xbb}:=\left(X^{1,t,\xbb}, \cdots, X^{n,t,\xbb} \right)$ be a  process s.t. $X^{i,t,\xbb}:=X^i$ with for each $s \in [0,t]$, $X^i_s=x^i$ and
    $$
        \mathrm{d}X^i_s=\overline{B}\left(s,X^i_s,\mu^n_s \right) \mathrm{d}s + \mathrm{d}W^i_s\;\mbox{with}\;\mu^n_s:=\frac{1}{n}\sum_{j=1}^n \delta_{X^j_s},\;\mbox{for}\;t \le s \le T
    $$
    where $(W^i)_{i \ge 1}$ is a sequence of independent Brownian motions. We also consider a sequence $(U^i)_{i \ge 1}$ satisfying: $U^i_0=X^{i}_0$ where $(X^i_0)_{i \ge 1}$ is an i.i.d sequence with distribution $\mu_0$ and, $\mathrm{d}U^i_t=B(t,U^i_t)\mathrm{d}t + \mathrm{d}W^i_t$. For each $n \ge 1$, we set 
    $$
        \mathrm{d}L^n_t=L^n_t \sum_{i=1}^n\left( B\left(t,X^{i,0, \Xbb^i_0}_t \right) - \overline{B}\left(t,X^{i,0, \boldsymbol{X}^i_0}_t, \nu^n_t \right) \right) \mathrm{d}W^i_t\;\mbox{with }\nu^n_t:=\frac{1}{n}\sum_{j=1}^n \delta_{X^{j,0,\Xbb^j_0}_t}
    $$
    and $\mathrm{d}\P^n:=L^n_T \mathrm{d}\P$. Since $\overline{B}$ and $B$ are bounded, by Girsanov's Theorem, we have $\Lc^{\P^n}\left( X^{1,0, \Xbb^1_0}, \cdots, X^{n,0, \Xbb^n_0} \right)=\Lc\left(U^1,\cdots,U^n \right)$. Since $F$ is bounded and continuous, it is simple to check that (see for instance \cite{lacker2018OnStrong} or \cite{djete2019general} for the propagation of chaos part) 
    \begin{align*}
        F \left( \Lc\left( X^{t,U_t}_T \right) \right) 
        =
        \lim_{n \to \infty} \E \left[ F \left( \frac{1}{n} \sum_{j=1}^n \delta_{X^{j,t, \Ubb_t}_T} \right) \right]
        =
        \lim_{n \to \infty} \E^{\P^n} \left[ \E \left[ F \left( \frac{1}{n} \sum_{j=1}^n \delta_{X^{j,0,\Xbb_0}_T} \right) \bigg| \Fc_t\right] \right]
    \end{align*}
    where $\Fc_t:=\sigma \left( W^i_{r}, X^i_0: r \le t, i \ge 1 \right)$.
    Notice that $(M^n_t)_{t \in [0,T]}:=\left( \E \left[ F \left( \frac{1}{n} \sum_{j=1}^n \delta_{X^{j,0,\Xbb_0}_T} \right) \bigg| \Fc_t\right] \right)_{t \in [0,T]}$ is a $\P$--martingale and 
    \begin{align*}
        \E \left[ F \left( \frac{1}{n} \sum_{j=1}^n \delta_{X^{j,0,\Xbb_0}_T} \right) \bigg| \Fc_t\right]
        =
        v^n\left(t,X^{1,0,\Xbb_0}_t,\cdots,X^{n,0,\Xbb_0}_t \right)\mbox{ with }v^n(t,\xbb):=\E \left[ F \left( \frac{1}{n} \sum_{j=1}^n \delta_{X^{j,t,\xbb}_T} \right)\right].
    \end{align*}
    Therefore,
    \begin{align*}
        \mathrm{d}M^n_t= \sum_{i=1}^n \partial_{x^i} v^n\left(t,X^{1,0,\Xbb_0}_t,\cdots,X^{n,0,\Xbb_0}_t \right) \mathrm{d}W^i_t.
    \end{align*}
    By the Clark--Ocone formula, see for instance \cite[Proposition 1.5]{nualart2006malliavin}, we have
    \begin{align*}
        \partial_{x^i} v^n \left(t,\xbb \right)= \sum_{j=1}^n\E \left[ J^{j,t,\xbb,i}_T \frac{1}{n}\partial_x \delta_mF \left( \frac{1}{n} \sum_{k=1}^n \delta_{X^{k,t,\xbb}_T} \right) \left( X^{j,t,\xbb}_T \right)\right]
    \end{align*}
    with $\nu^{n,t,\xbb}_s:=\frac{1}{n} \sum_{k=1}^n \delta_{X^{k,t,\xbb}_s}$, for each $s \in [0,t]$, $J^{j,t,\xbb,i}_s=\mathbf{1}_{i = j}$ and, for each $s \in [t,T]$,
    \begin{align*}
        \mathrm{d}J^{j,t,\xbb,i}_s=J^{j,t,\xbb,i}_s \overline{B}' \left(s, X^{j,t,\xbb}_s, \nu_s^{n,t,\xbb} \right) \mathrm{d}s + \sum_{k=1}^n J^{k,t,\xbb,i}_s \frac{1}{n} \partial_y \delta_m \overline{B} \left(s, X^{j,t,\xbb}_s, \nu_s^{n,t,\xbb} \right)(X^{k,t,\xbb}_s) \mathrm{d}s.
    \end{align*}
    As a consequence,
    \begin{align*}
         \mathrm{d}M^n_t= \frac{1}{n}\sum_{i=1}^n  \sum_{j=1}^n\E \left[ J^{j,t,\Xbb_t,i}_T \partial_x \delta_mF \left( \frac{1}{n} \sum_{k=1}^n \delta_{X^{k,t,\Xbb_t}_T} \right) \left( X^{j,t,\Xbb_t}_T \right)\bigg| \Fc_t\right]\mathrm{d}W^i_t.
    \end{align*}
    Let $t \in [0,T]$ be fixed. We introduce the sequence $\left( \Qr^n \right)_{n \ge 1} \subset \Pc \left( \Pc(\Cc^3 ) \x \Cc^3 \right)$ by
    \begin{align*}
        \Qr^n:=\frac{1}{n} \sum_{i=1}^n \P \circ  \left( \muh^{n,i}, J^{i,t,\Ubb_t,i}, X^{i,t,\Ubb_t}, U^i_{t \wedge \cdot} \right)^{-1}\;\mbox{where}\;\muh^{i,n}:=\frac{1}{n-1} \sum_{k \neq i} \delta_{\left( n J^{k,t,\Ubb_t,i}, X^{k,t,\Ubb_t}, U^j_{t \wedge \cdot} \right)}.
    \end{align*}
    \textbf{Step 2 : convergence and characterization of $(\Qr^n)_{n \ge 1}$ } Let $p>1$, by using the conditions over $\overline{B}$, there exists a constant $C>0$ (independent of $n$), s.t. for any $i,j$, any $t \in [0,T]$,
    \begin{align*}
        |J^{j,t,\Ubb_t,i}_t|^p \le C \left( \mathbf{1}_{i=j} + \int_0^t \sup_{r \le s}|J^{j,t,\Ubb_t,i}_r|^p \mathrm{d}s + \frac{1}{n} \sum_{k=1}^n\int_0^t  \sup_{r \le s}|J^{k,t,\Ubb_t,i}_r|^p \right).
    \end{align*}
    By using Gronwall's Lemma and the fact that $\mathbf{1}_{i=j}$ is only equal to $1$ when $i=j$ , we deduce that
    \begin{align*}
        \sup_{n \ge 1}\sup_{1 \le i \le n}\frac{1}{n} \sum_{j=1}^n \E \left[ \sup_{s \in [0,T]} |J^{j,t,\Ubb_t,i}_s|^p \left( n^p\mathbf{1}_{i \neq j} + \mathbf{1}_{i = j} \right) \right] <\infty,
    \end{align*}
    for any $p >1$. Consequently, the sequence $(\Qr^n)_{n \ge 1}$ is relatively compact in $\Wc_p$ for any $p>1$ (see for instance \cite[Proposition B.1]{Lacker_carmona_delarue_CN} . Let $\Qr=\P \circ (\muh,J,X,U)^{-1}$ be the limit of a convergent sub--sequence. Notice that
    \begin{align*}
        J^{i,t,\xbb,i}_t=1,\;\mathrm{d}J^{i,t,\xbb,i}_s=J^{i,t,\xbb,i}_s \overline{B}' \left(s, X^{i,t,\xbb}_s, \nu_s^{n,t,\xbb} \right) \mathrm{d}s &+ \frac{n-1}{n^2} \frac{1}{n-1} \sum_{k\neq i} nJ^{k,t,\xbb,i}_s  \partial_y \delta_m \overline{B} \left(s, X^{i,t,\xbb}_s, \nu_s^{n,t,\xbb} \right)(X^{k,t,\xbb}_s) \mathrm{d}s 
        \\
        &+ \frac{1}{n} J^{i,t,\xbb,i}_s  \partial_y \delta_m \overline{B} \left(s, X^{i,t,\xbb}_s, \nu_s^{n,t,\xbb} \right)(X^{i,t,\xbb}_s) \mathrm{d}s
    \end{align*}
    and for $i\neq j$, $nJ^{j,t,\xbb,i}_t=0$
    \begin{align*}
        \;\mathrm{d}nJ^{j,t,\xbb,i}_s=nJ^{j,t,\xbb,i}_s \overline{B}' \left(s, X^{j,t,\xbb}_s, \nu_s^{n,t,\xbb} \right) \mathrm{d}s &+ \frac{n(n-1)}{n^2} \frac{1}{n-1} \sum_{k\neq i} nJ^{k,t,\xbb,i}_s  \partial_y \delta_m \overline{B} \left(s, X^{j,t,\xbb}_s, \nu_s^{n,t,\xbb} \right)(X^{k,t,\xbb}_s) \mathrm{d}s 
        \\
        &+ J^{i,t,\xbb,i}_s  \partial_y \delta_m \overline{B} \left(s, X^{j,t,\xbb}_s, \nu_s^{n,t,\xbb} \right)(X^{i,t,\xbb}_s) \mathrm{d}s.
    \end{align*}
    By using classical weak convergence techniques for interacting particles (see \cite[Proposition 5.1.]{lacker2017limit} and \cite[Proposition 4.17.]{djete2019general}, and its associated proofs ), the process $(\mu,J,X,U)$ verifies: $\Lc(X,U_{t \wedge \cdot})=\Lc(X^{t,U_t},U_{t \wedge \cdot})$, $\P$ a.e. for each $s \in [0,t]$, $J_s=1$ and $s \in [t,T]$,
    \begin{align*}
        \mathrm{d}J_s=J_s \overline{B}' (s,X_s,\nu_s) \mathrm{d}s\;\mbox{with}\;\nu_s=\Lc(X^{t,U_t}_s),
    \end{align*} 
    and  $\P$--a.e. $\om \in \Om$, $\muh(\om)=\Lc^{\hat\mu(\om)}(\widehat{J},\widehat{X}, \widehat{U})$ where $(\widehat{J},\Xh,\widehat{U})$ are the canonical processes that satisfy: $\Lc^{\hat\mu(\om)}(\widehat{X}, \widehat{U})=\Lc(X^{t,U_t},U_{t \wedge \cdot})$,  for all $s\in [0,t]$, $\widehat{J}_s=0,$ and for $s \in [t,T]$,
    \begin{align} \label{eq:measure_derivative_process}
        \mathrm{d}\widehat{J}_s=\widehat{J}_s \overline{B}'\left(s,\Xh_s,\nu_s \right) \mathrm{d}s + \int_{\Cc \x \Cc}  j_s \partial_y \delta_m \overline{B} (s, \Xh_s,\nu_s)(x_s) \muh(\om)(\mathrm{d}j, \mathrm{d}x,\Cc)  \mathrm{d}s + J_s(\om) \partial_y \delta_m \overline{B}(s, \widehat{X}_s,\nu_s)(X_s(\om)) \mathrm{d}s
    \end{align}
    \textbf{~~~~Step 2.1 : analysis of Equation \eqref{eq:measure_derivative_process} }
    Let us observe that: for any $(j,x)=(j_s,x_s)_{s \in [0,T]}$, there exists a unique McKean--Vlasov process $(V,S)$ (see for instance \cite[Theorem A.3]{djete2019mckean}) satisfying: for each $s \in [0,t]$, $V_t=0$, $S_t=U_t$ and, for each $s \in [0,t]$,
    \begin{align*}
        \mathrm{d}V_s=V_s \overline{B}'\left(s,S_s,\nu_s \right) \mathrm{d}s + \int_{\Cc \x \Cc} v_s \partial_y \delta_m \overline{B} (s, S_s,\nu_s)(y_s) \eta(\mathrm{d}v, \mathrm{d}y) + j_s \partial_y \delta_m \overline{B}(s, S_s,\nu_s)(x_s),\;\eta:=\Lc(V,S)
    \end{align*}
    with $\mathrm{d}S_s=\overline{B}(s,S_s,\nu_s)\mathrm{d}s+ \mathrm{d}W_s,\;\nu_s=\Lc(S_s).$ 

    \medskip
    Notice that, by uniqueness, $\Lc(S)=\Lc(X^{t,U_t})$, then $\nu_s=\Lc(X^{t,U_t}_s)$. Therefore, since $(j_s)_{s \in [0,T]}$, $(x_s)_{s \in [0,T]}$, $(W^t_s:=W_s-W_t)_{s \in [0,T]}$ and $U_{t \wedge \cdot}$ are the only inputs of the equation, there exists a Borel map $\widetilde{\Gamma}:\Cc^4  \to \Cc$ s.t. $V=\widetilde{\Gamma} (j,x, W^t,U_{t \wedge \cdot})$ a.e. But $W^t$ is a map of $S$, we can rewrite $V=\Gamma (j,x, S,U_{t \wedge \cdot})$ for some appropriate Borel map $\Gamma$. 
    Since the pair of canonical processes $( \widehat{J}, \widehat{X})$ satisfies the same equation as $(V, S)$, it follows that, for $\P$--a.e.\ $\omega$, the identity $\widehat{J} = \Gamma \left( J(\omega), X(\omega), \widehat{X}, \widehat{U} \right)$ holds $\widehat{\mu}(\omega)$--a.e.

\medskip    
    It is straightforward to see that $J$ is a Borel map of $X$ i.e. $\left(J_s={\rm exp}\left(\int_t^s \overline{B}'(r,X_r,\nu_r)\mathrm{d}r \right) \right)_{s \in [0,T]}=:L(X)$, therefore 
    \begin{align} \label{eq:mu_x_relation}
        \muh(\om)&=\Lc^{\hat \mu (\om)}\left(\widehat{J}, \Xh, \widehat{U} \right) \nonumber
        \\
        &=\Lc^{\hat \mu (\om)}\left(\Gamma \left( J(\omega), X(\omega), \widehat{X}, \widehat{U} \right), \Xh, \widehat{U} \right) \nonumber
        \\
        &=\Lc\left(\Gamma \left( J(\omega), X(\omega), X^{t,U_t}, U \right), X^{t,U_t}, U \right)=\Lc\left(\Gamma \left( L(X(\omega)), X(\omega), X^{t,U_t}, U \right), X^{t,U_t}, U \right)=K(X(\om))
    \end{align}
    for some Borel map $K$. The maps $\Gamma$, $L$ and $K$ are independent of the measure $\Qr$, then $\Qr=\Lc^\P(\muh,J,X,U)= \Lc^\P(K(X),L(X),X,U)$ and $\Lc(X,U)=\Lc(X^{t,U_t},U)$. This is true for any sub--sequence of $(\Qr^n)_{n \ge 1}$, as a consequence the entire sequence converges to $\Lc(K(X^{t,U_t}),L(X^{t,U_t}),X^{t,U_t},U)$. 
    
\medskip
    \textbf{~~~~Step 2.2 : representation of the limit }  Set 
    $$
        \overline{\Qr}:= \E \left[ \Lc^{\hat \mu} \left( \widehat{J}, \widehat{X}, \widehat{U} \right)\delta_{\left(J,\;X,\;U \right)}\right] \in \Pc(\Cc^6).
    $$
    With the observation of \eqref{eq:mu_x_relation} i.e. $\muh$ is a map of $X$, by a straightforward manipulation, we get 
    $$
        \overline{\Qr}=\Lc\left(\overline{J}^{t,U_t,\overline{U}_t},\overline{X}^{t,\overline{U}_t},\overline{U}_{t \wedge \cdot},J^{t,U_t},X^{t,U_t},U_{t \wedge \cdot} \right)
    $$
    where: $\overline{X}^{t,\overline{U}_t}_t=\overline{U}_t$, ${X}^{t,{U}_t}_t$, $\nu_s=\Lc({X}^{t,{U}_t}_s)$,
    \begin{align*}
        \mathrm{d}\overline{X}^{t,\overline{U}_t}_s=\overline{B}\left(s, \overline{X}^{t,\overline{U}_t}_s, \nu_s  \right)\mathrm{d}s + \mathrm{d}\overline{W}_s,\;\mathrm{d}{X}^{t,{U}_t}_s=\overline{B}\left(s, {X}^{t,{U}_t}_s, \nu_s  \right)\mathrm{d}s + \mathrm{d}{W}_s,
    \end{align*}
    $J^{t,U_t}_t=1$, $\overline{J_t}^{t,U_t,\overline{U}_t}=0$, $\muh_s=\Lc\left(\overline{J}^{t,U_t,\overline{U}_t}_s,\overline{X}^{t,\overline{U}_t}_s\mid X^{t,U_t}_{s \wedge \cdot} \right)$, $\mathrm{d}J^{t,U_t}_s=J^{t,U_t}_s \overline{B}' (s,X^{t,U_t}_s,\nu_s) \mathrm{d}s$,
    \begin{align*}
        \;\mathrm{d}\overline{J}^{t,U_t,\overline{U}_t}_s=\overline{J}^{t,U_t,\overline{U}_t}_s \overline{B}'\left(s,\overline{X}^{t,\overline{U}_t}_s,\nu_s \right) \mathrm{d}s + \int_{\R \x \R} j \partial_y \delta_m \overline{B} (s, \overline{X}^{t,\overline{U}_t}_s,\nu_s)(x) \muh_s(\mathrm{d}j,\mathrm{d}x) + J^{t,U_t}_s \partial_y \delta_m \overline{B}(s, \overline{X}^{t,\overline{U}_t}_s,\nu_s)(X^{t,U_t}_s)
    \end{align*}
    with $(\overline{W},W)$ is an $\R^2$--Brownian motion and, $U$ and $\overline{U}$ are independent with the same law $\Lc(U)$.

    \medskip
    \textbf{Step 3 : verification of the limit} By using, what we proved in the previous step, we have
    \begin{align*}
        &\lim_{n \to \infty}\E^{\P^n} \left[M^n_s \right]
        =
        \lim_{n \to \infty} \E^{\P^n} \left[ M^n_T - \frac{1}{n}\sum_{i=1}^n \sum_{j=1}^n\int_s^T \E \left[ J^{j,t,\Xbb_t,i}_T \partial_x \delta_mF \left(\frac{1}{n} \sum_{k=1}^n \delta_{X^{k,t,\Xbb_t}_T} \right) \left( X^{j,t,\Xbb_t}_T \right) \bigg| \Fc_t\right] \mathrm{d}W^i_t\right]
        \\
        &=\lim_{n \to \infty}   \E \left[ F \left( \frac{1}{n} \sum_{j=1}^n \delta_{ U^j_T} \right) \right] 
        \\
        &~~~~~~-\lim_{n \to \infty}\frac{1}{n}\sum_{i=1}^n \sum_{j=1}^n\int_s^T \E \left[ J^{j,t,\Ubb_t,i}_T \partial_x \delta_mF \left( \frac{1}{n}\sum_{k=1}^n \delta_{X^{k,t,\Ubb_t}_T} \right) \left( X^{j,t,\Ubb_t}_T \right) \left( B(t,U^i_t)- \overline{B}(t, U^i_t,\mu^n_t) \right) \right] \mathrm{d}t  
        \\
        &=\lim_{n \to \infty}   \E \left[ F \left( \frac{1}{n} \sum_{j=1}^n \delta_{ U^j_T} \right) \right] -\lim_{n \to \infty}\frac{1}{n}\sum_{i=1}^n \int_s^T \E \left[ J^{i,t,\Ubb_t,i}_T \partial_x \delta_mF \left(\frac{1}{n} \sum_{k=1}^n \delta_{X^{k,t,\Ubb_t}_T} \right) \left( X^{i,t,\Ubb_t}_T \right) \left( B(t,U^i_t)- \overline{B}(t, U^i_t,\mu^n_t) \right) \right] \mathrm{d}t
        \\
        &~~~~~~-\lim_{n \to \infty}\frac{1}{n^2}\sum_{i=1}^n \sum_{j\neq i}\int_s^T \E \left[ n J^{j,t,\Ubb_t,i}_T \partial_x \delta_mF \left( \frac{1}{n} \sum_{k=1}^n \delta_{X^{k,t,\Ubb_t}_T} \right) \left( X^{j,t,\Ubb_t}_T \right) \left( B(t,U^i_t)- \overline{B}(t, U^i_t,\mu^n_t) \right) \right] \mathrm{d}t
        \\
        &=F\left(\Lc(U_T) \right) - \E \left[ \int_s^T J^{t,U_t}_T \partial_x \delta_m F\left( \Lc\left( X^{t,U_t}_T \right)  \right) \left( X^{t,U_t}_T \right) \left( {B} (t,U_t) - \overline{B}(t,U_t,\mu_t) \right) \mathrm{d}t \right]
        \\
        &~~~~~~~~~~~~~~~~~-\E \left[ \int_s^T \overline{J}^{t,U_t,\overline{U}_t}_T \partial_x \delta_m F\left( \Lc\left( X^{t,U_t}_T \right)  \right) \left( \overline{X}^{t,\overline{U}_t}_T \right) \left( {B} (t,U_t) - \overline{B}(t,U_t,\mu_t) \right) \mathrm{d}t \right].
    \end{align*}
    Let us define $\beta(t,x):=\E \left[ \delta_m F\left( \Lc\left( X^{t,U_t}_T \right)  \right) \left( X^{t,x}_T \right) \right].$ Notice that
    $$
        \partial_x\beta(t,x)=\E \left[ J^{t,x}_T \;\partial_x\delta_m F\left( \Lc\left( X^{t,U_t}_T \right)  \right) \left( X^{t,x}_T \right) \right].
    $$
    Under the assumptions (see \cite{matiichuk1965parabolic} or \cite{KUSUOKA2017359} for instance ), there exists a Borel map $p:[0,T] \x \R \x [0,T] \x \R \to \R_+$ s.t. $\int_\R p(s,y,t,x) \mathrm{d}y=1$ for each $(s,t,x)$, and $\Lc(X^{t,x}_s)(\mathrm{d}y)=p(s,y,t,x)\mathrm{d}y$. In addition $p$ is differentiable in $x$ with $\int_0^s\int_\R |\partial_x p(s,y,t,x)|\mathrm{d}y\;\mathrm{d}t<\infty$, and we have
    \begin{align*}
        \partial_x\beta(t,x)
        =
        \int_\R \delta_m F\left( \Lc\left( X^{t,U_t}_T \right)  \right) \left( y \right)\; \partial_xp(T,y,t,x) \mathrm{d}y
        =
        \E \left[ \delta_m F\left( \Lc\left( X^{t,U_t}_T \right)  \right) \left( X^{t,x}_T \right) \frac{\partial_xp(T,X^{t,x}_T,t,x)}{p(T,X^{t,x}_T,t,x)} \right].
    \end{align*}
    Consequently,
    \begin{align*}
         F \left( \Lc\left( X^{t,U_t}_T \right) \right)=F\left(\Lc(U_T) \right) &+ \E \left[ \int_t^T \delta_m F\left( \Lc\left( X^{r,U_r}_T \right)  \right) \left( X^{r,U_r}_T \right) \frac{\partial_xp(T,X^{r,U_r}_T,r,U_r)}{p(T,X^{r,U_r}_T,r,U_r)} \left( \overline{B} (r,U_r,\mu_r) - B(r,U_r) \right) \mathrm{d}r \right]
         \\
         &+ \E \left[ \int_t^T \overline{J}^{r,U_r,\overline{U}_r}_T \partial_x \delta_m F\left( \Lc\left( X^{r,U_r}_T \right)  \right) \left( \overline{X}^{r,\overline{U}_r}_T \right) \left( \overline{B} (r,U_r,\mu_r) - B(r,U_r) \right) \mathrm{d}r \right].
    \end{align*}
    Notice that since all the maps are assumed to be continuous bounded, we can check that the previous equality is also true if $\mu_0 \in \Pc(\R)$ (and not $\mu_0 \in \Pc_p(\R)$ for any $p$ ) by approximating any distribution in $\Pc(\R)$ by a sequence of probability with compact support (see also the techniques of \cite[Theorem A.2.]{djete2019general}).
\end{proof}

In the special case where the function $\overline{B} : [0,T] \times \mathbb{R} \times \mathcal{P}(\mathbb{R}) \to \mathbb{R}$ does not depend on the measure argument—i.e., when there exists a Borel measurable map $\widetilde{B} : [0,T] \times \mathbb{R} \to \mathbb{R}$ such that $\overline{B}(t,x,m) = \widetilde{B}(t,x)$ for all $(t,x,m)$—the formulation simplifies. For convenience, we retain the notation $\overline{B}$ in place of $\widetilde{B}$. In this context, we only require the functional $F : \mathcal{P}(\mathbb{R}) \to \mathbb{R}$ to admit a linear functional derivative with respect to the measure, with $(m,x) \mapsto \delta_m F(m)(x)$ being bounded and continuous in $m$, but not necessarily differentiable or continuous in the spatial variable $x$.


\begin{corollary} \label{corollary:FK_measure_simplified}
    For any $\mu \in \Lc$ associated to $B$, and $(U_s)_{s \in [0,T]}$ the process s.t. $\Lc(U_s)=\mu_s$. We have for each $t \in [0,T]$,
    \begin{align} \label{eq:FK_measure_simplified}
         F \left( \Lc\left( X^{t,U_t}_T \right) \right)=F\left(\Lc(U_T) \right) + \E \left[ \int_t^T \delta_m F\left( \Lc\left( X^{r,U_r}_T \right)  \right) \left( X^{r,U_r}_T \right) \frac{\partial_xp(T,X^{r,U_r}_T,r,U_r)}{p(T,X^{r,U_r}_T,r,U_r)} \left( \overline{B} (r,U_r) - B(r,U_r) \right) \mathrm{d}r \right].
    \end{align}
\end{corollary}

\begin{proof}
    In that case, we can check that $\overline{J}^{t,U_t,\overline{U}_t}_s=0$ a.e. for all $s \in [0,T]$. Indeed, since there is no measure dependence, the equation satisfied by $\left(\overline{J}^{t,U_t,\overline{U}_t}_s \right)_{s \in [0,T]}$ is: for each $s \in [0,t]$, $\overline{J}^{t,U_t,\overline{U}_t}_s=0$ and for each $s \in [t,T],$  $\;\mathrm{d}\overline{J}^{t,U_t,\overline{U}_t}_s=\overline{J}^{t,U_t,\overline{U}_t}_s \overline{B}'\left(s,\overline{X}^{t,\overline{U}_t}_s,\nu_s \right) \mathrm{d}s$. The unique solution is $\overline{J}^{t,U_t,\overline{U}_t}_s=0$ a.e. for all $s \in [0,T]$. Consequently, \eqref{eq:FK_measure_simplified} is true with the assumption of \Cref{prop:FK_measure_general}. For each $\varepsilon >0$, we set $F_\varepsilon(m)=F(\Lc(\varepsilon W_T+S))$ where $S$ is a random variable independent of $W$ with $\Lc(S)=m$. By simple computations, we get $\delta_m F_\varepsilon(m)(x)=\E\left[\delta_m F(\Lc(\varepsilon W_T +S))(\varepsilon W_T + x)\right]$ and $\partial_x \delta_m F_\varepsilon(m)(x)=\E\left[\delta_m F(\Lc(\varepsilon W_T +S))(\varepsilon W_T + x)\frac{W_T}{\varepsilon T}\right]$. We verify that $\lim_{\varepsilon \to 0}F_\varepsilon(m)=F(m)$. Notice that, since $\partial_x p$ is integrable (see \Cref{remark:integrability_Z}), we get 
\begin{align*}
    &\sup_{\varepsilon > 0}\left| \int_\R\delta_m F_\varepsilon(m) \left(y \right) \partial_xp(T,y,r,x)\mathrm{d}y   - \int_\R\delta_m F_\varepsilon(m) \left(y \right) \partial_xp(T,y,r,x)\mathbf{1}_{|y| \le K}\mathrm{d}y  \right|
    \\
    &= \sup_{\varepsilon > 0}\left| \int_\R\delta_m F_\varepsilon(m) \left(y \right) \partial_xp(T,y,r,x)\mathbf{1}_{|y| > K}\mathrm{d}y  \right| \le \sup_{(x,m)} |\delta_m F(m)(x)|  \int_\R \left|\partial_xp(T,y,r,x)\right|\mathbf{1}_{|y| > K}\mathrm{d}y   \to_{K \to \infty}0.
\end{align*}
By dominated convergence Theorem, we obtain 
\begin{align*}
    &\lim_{K \to \infty} \lim_{\varepsilon \to 0} \int_\R\delta_m F_\varepsilon(m) \left(y \right) \partial_xp(T,y,r,x)\mathbf{1}_{|y| \le K}\mathrm{d}y
    \\
    &=\lim_{K \to \infty} \lim_{\varepsilon \to 0} \E \left[\int_\R\delta_m F(\Lc(\varepsilon W_T + S )) \left(y \right) \partial_xp(T,y-\varepsilon W_T,r,x)\mathbf{1}_{|y-\varepsilon W_T| \le K}\mathrm{d}y \right]= \int_\R\delta_m F(m) \left(y \right) \partial_xp(T,y,r,x)\mathrm{d}y.
\end{align*}
Then, $\lim_{\varepsilon \to 0} \int_\R\delta_m F_\varepsilon(m) \left(y \right) \partial_xp(T,y,r,x)\mathrm{d}y=\int_\R\delta_m F(m) \left(y \right) \partial_xp(T,y,r,x)\mathrm{d}y$. Since $\partial_x \delta_m F_\varepsilon$ is continuous and bounded, we can apply \eqref{eq:FK_measure_simplified} for $F_\varepsilon$, by letting $\varepsilon \to 0$, we deduce the result for $F$ .
\end{proof}

{\color{black}
\begin{remark} \label{remark:integrability_Z}

By {\rm \citeauthor*{matiichuk1965parabolic} \cite{matiichuk1965parabolic}} and {\rm \citeauthor*{porper1992properties} \cite{porper1992properties}} $($see also {\rm \citeauthor*{KUSUOKA2017359} \cite{KUSUOKA2017359}}$)$,  there exist constants \( C > 0 \) and \( \gamma > 0 \) $($depending on the Lipschitz continuity and boundedness of \( \overline{B} \)$)$ such that for \( t \ge r \) and any \( (x, y) \in \mathbb{R}^2 \), we have \( |\partial_x p(t, y, r, x)| \le \frac{C}{|t - r|} \exp\left( -\frac{\gamma |x - y|^2}{t - r} \right) \). As a result, one can verify that \( \mathbb{E} \left[ \int_0^T \int_{\mathbb{R}} |\partial_x p(t, y, r, U_r)| \, \mathrm{d}y \, \mathrm{d}t \right] < \infty \). We deduce that \( \mathbb{E} \left[ \int_0^T |Z(t, U_t, \mu_t)| \, \mathrm{d}t \right] < \infty \) where the process $Z$ is \( Z(r, x, \mu_r) := \mathbb{E} \left[ \delta_m F\left( \mathcal{L}(X_T^{r, U_r}) \right)\left(X_T^{r, U_r} \right) \frac{ \partial_x p(T, X_T^{r, U_r}, r, U_r) }{ p(T, X_T^{r, U_r}, r, U_r) } \,\Big|\, U_r = x \right] \).
\end{remark}
}

\subsection{An existence result}

We provide an existence result for an equation in which the unknown is a functional over the Wasserstein space of probability measures.

\medskip
Let $\overline{f} : [0,T] \times \Pc(\R) \times \R \to \R$ and $\overline{P} : [0,T] \times \R \to \R$ be two bounded Borel measurable functions. We assume that the function $\overline{f}$ is differentiable with respect to the variable $y$, with derivative denoted by $\overline{f}'$, and admits a linear functional derivative with respect to the measure argument, denoted by $\delta_m \overline{f}$. Furthermore, the map 
\[
(t, m, y, z) \mapsto \left( \overline{f}(t, m, y), \overline{f}'(t, m, y), \delta_m \overline{f}(t, m, y)(z) \right)
\]
is bounded and Lipschitz continuous in the variable $y$, uniformly in $(t, m, z)$. We denote by ${\rm Lip}(h)$ the Lipschitz constant of a function $h \in \{ \overline{f}, \overline{f}', \delta_m \overline{f} \}$. The notation $|h|_\infty$ will be used to refer to the uniform norm of $h$, regardless of the underlying domain.
Furthermore, we consider a continuous bounded function $\psi : \Pc(\R) \to \R$ which admits a bounded linear functional derivative $\delta_m \psi$.


\medskip
Given $(t,x) \in [0,T] \x \R$, we denote by $(X^{t,x}_s)_{s \in [0,T]}$ the process verifying: for each $s \in [0,t]$, $X^{t,x}_s=x$ and, for any $s \in [t,T]$,
\begin{align*}
    \mathrm{d}X^{t,x}_s= \overline{P}(s, X^{t,x}_s)\mathrm{d}s +\mathrm{d}W_s
\end{align*}
where $W$ is an $\R$--valued $\F$--Brownian motion.
For any $m \in \Pc(\R)$, we denote by $U$ the random variable independent of $W$ s.t. $\Lc(U)=m$. 
\begin{proposition} \label{prop:existence:Y}
    There exists a continuous bounded map $Y:[0,T] \x \Pc(\R) \to \R$ verifying: for any $(t,m),$
    \begin{align*}
        Y(t,m)
        =
        \psi\left( \Lc(X^{t,U}_T) \right) + \int_t^T \overline{f} \left( r, \Lc(X^{t,U}_r), Y\left(r, \Lc(X^{t,U}_r) \right)\; \right)\;\mathrm{d}r.
    \end{align*}
    In addition, if $T$ is small enough i.e.
    \begin{align*}
        T(C+|\overline{f}'|_\infty)<1\;\mbox{and}\;T\;{\rm Lip}(\delta_m \overline{f}) + T\;L_T {\rm Lip}(\overline{f}') + T\;|\overline{f}'|_\infty + T\;{\rm Lip}(\overline{f})< 1
    \end{align*}
    where $\frac{|\psi|_\infty+TC}{1-TC} + \frac{||\delta_m \psi|_\infty| + T |\delta_m \overline{f}|_\infty}{1- T |\overline{f}'|_\infty } = L_T$ then    for each $t$, the map $m \mapsto Y(t,m)$ admits a linear functional derivative $\delta_m Y$ verifying $\sup_{(t,x,m)}|\delta_m Y(t,m)(x)|< \infty $.
\end{proposition}

\begin{proof}
    Let $\alpha >0$. Let us start by introducing the set $\Rc_\alpha$. We will say that a Borel map $y:[0,T] \x \Pc(\R) \to \R$ belongs to $\Rc_\alpha$ if $\|y\|_\alpha:=\int_0^T e^{\alpha t} \sup_m|y(t,m)|\mathrm{d}t < \infty$. We introduce the map $\Gamma:\Rc_\alpha \to\Rc_\alpha$ by
    \begin{align*}
        \Gamma_y(t,m):= \psi\left( \Lc(X^{t,U}_T) \right) + \int_t^T \overline{f} \left( r, \Lc(X^{t,U}_r), y\left(r, \Lc(X^{t,U}_r) \right)\; \right)\;\mathrm{d}r.
    \end{align*}
    Notice that, since $\psi$ and $\overline{f}$ are bounded, $\Gamma$ is well--defined i.e. for any $y \in \Rc_\alpha$, $\Gamma_y \in \Rc_\alpha$. There exists $C>0$ (independent of $\alpha$ ), for any $(y,y')$,
    \begin{align*}
        &\left|\Gamma_y-\Gamma_{y'} \right|(t,m)
        =\left|\int_t^T \overline{f} \left( r, \Lc(X^{t,U}_r), y\left(r, \Lc(X^{t,U}_r) \right)\; \right)-\overline{f} \left( r, \Lc(X^{t,U}_r), y'\left(r, \Lc(X^{t,U}_r) \right)\; \right)\;\mathrm{d}r \right|
        \\
        &\le C \int_t^T \left| y\left(r, \Lc(X^{t,U}_r) \right)- y'\left(r, \Lc(X^{t,U}_r)\; \right)\right|\;\mathrm{d}r\le  C \int_t^T \sup_m \left| y\left(r, m \right)- y'\left(r, m \right)\right|\;\mathrm{d}r.
    \end{align*}
    Then, we obtain
    \begin{align*}
        \|\Gamma_y-\Gamma_{y'}\|_\alpha \le C \int_0^T e^{\alpha t} \int_t^T \sup_m \left| y\left(r, m \right)- y'\left(r, m \right)\right|\;\mathrm{d}r\; \mathrm{d}t \le \frac{C}{\alpha} \|y-y'\|_\alpha.
    \end{align*}
    Also, by similar way, we get
    \begin{align*}
        \|\Gamma_y\|_\alpha \le |\psi|_\infty \frac{e^{\alpha T}}{\alpha} + \frac{C}{\alpha} \left( e^{\alpha T} + \|y\|_\alpha \right). 
    \end{align*}
    Let $\alpha > C$ and $L \in \R_+$ s.t. $L > \frac{e^{\alpha T} \left(|\psi|_\infty + C\right)}{\alpha-C}$. We can apply the Banach fixed point Theorem and deduce that there exists $Y \in \Rc_\alpha$ s.t. $\|Y\|_{\alpha} \le L$ and $\Gamma_{Y}=Y$. In addition, with $Y^0=0$, by defining $Y^{n+1}=\Gamma_{Y^n}$ for any $n \ge 0$, we have $\lim_{n \to \infty} \|Y^n-Y\|_\alpha=0$. This allows us to verify the continuity of $Y$.

    \medskip
    Let us observe that, if $y$ admits a bounded linear functional derivative $\delta_m y$, we have
    \begin{align*}
        \delta_m \Gamma_y(t,m)(x)=\E\bigg[ \delta_m \psi \left(\Lc(X^{t,U}_T) \right)\left(X^{t,x}_T \right) &+ \int_t^T \delta_m \overline{f} \left( r, \Lc(X^{t,U}_r), y (r, \Lc(X^{t,U}_r)) \right) (X^{t,x}_r ) \;\mathrm{d}r
        \\
        &+ \int_t^T \overline{f}' \left( r, \Lc(X^{t,U}_r), y (r, \Lc(X^{t,U}_r)) \right) \delta_m y(r, \Lc(X^{t,U}_r))(X^{t,x}_r) \;\mathrm{d}r\bigg].
    \end{align*}
    We can check that $|\delta_m \Gamma_y|_\infty \le |\delta_m \psi|_\infty +T\;|\delta_m \overline{f}|_\infty + T\;|\overline{f}'|_\infty |\delta_m y|_\infty$ and also, $|\Gamma_y|_\infty \le |\psi|_\infty + T\;C\left(1+|y|_\infty \right)$. 
    For another $y'$ admitting a bounded linear functional derivative,
    \begin{align*}
        \left| \delta_m \left(\Gamma_y-\Gamma_{y'} \right) \right|_\infty \le T\;{\rm Lip}(\delta_m \overline{f})|y-y'|_\infty + T\;|\delta_m y|_\infty {\rm Lip}(\overline{f}')|y-y'|_\infty + T\;|\overline{f}'|_\infty |\delta_m \left( y - y' \right)|_\infty
    \end{align*}
    and $|\Gamma_y-\Gamma_{y'}|_\infty \le T\;{\rm Lip}(\overline{f})|y-y'|_\infty$. Let $T$ be s.t. $T(C+|\overline{f}'|_\infty)<1$ and $L_T >0$ be s.t. $\frac{|\psi|_\infty+TC}{1-TC} + \frac{||\delta_m \psi|_\infty| + T |\delta_m \overline{f}|_\infty}{1- T |\overline{f}'|_\infty } \le L_T$. If $|\delta_m y|_\infty+ |y|_\infty \le L_T$ then $|\delta_m \Gamma_y|_\infty+ |\Gamma_y|_\infty \le L_T$. In addition, if $ T\;{\rm Lip}(\delta_m \overline{f}) + T\;L_T {\rm Lip}(\overline{f}') + T\;|\overline{f}'|_\infty + T\;{\rm Lip}(\overline{f})< 1$ then the map $\Gamma$ is a contraction in $\Sc:=\left\{y:\;|y|_\infty + |\delta_m y|_\infty \le L_T \right\}$ for the norm $\|y\|_\infty:=|y|_\infty + |\delta_m y|_\infty$. By Banach fixed point theorem there exists $Y \in \Sc$ verifying $\Gamma_Y=Y$ and $\delta_m \Gamma_Y =\delta_m Y$.
\end{proof}

\end{appendix}

\bibliographystyle{plain}


\bibliography{BSDE_McKVlasov_arxivVersion}




\end{document}